\documentclass[12pt]{amsart}
\usepackage[a4paper,scale=.781,centering]{geometry}  
\usepackage{amsaddr}
\usepackage{mathrsfs,amssymb,amsmath}
\usepackage[mathlines]{lineno}
\usepackage[dvipsnames]{xcolor}
\usepackage[
pagebackref,
colorlinks,citecolor=BlueViolet,urlcolor=BlueViolet,linkcolor=BlueViolet]{hyperref}

\usepackage{graphicx}
\usepackage{subfig}
 \theoremstyle{plain}
\newtheorem{theorem}{Theorem}

\newtheorem{proposition}{Proposition}
\newtheorem{lemma}[proposition]{Lemma}

\theoremstyle{definition}

\newtheorem*{remark}{Remark}
\newtheorem{openproblem}{Open Problem}

\newcommand{\cadlag}{c\`adl\`ag}

\newcommand{\re}{\ensuremath{\mathbb{R}}}
\newcommand{\paren}[1]{\ensuremath{\left( #1\right) }}
\newcommand{\F}{\ensuremath{\mathscr{F}}}

\newcommand{\p}{\mathbb{P}}

\newenvironment{lesn}{\begin{linenomath}\begin{equation*}}{\end{equation*}\end{linenomath}}
\newcommand{\imf}[2]{\ensuremath{#1\!\paren{#2}}}
\newcommand{\se}{\ensuremath{\mathbb{E}}}
\newcommand{\cond}[2]{\left.\vphantom{#2}#1\ \right| #2}
\DeclareMathOperator{\cb}{CB}
\DeclareMathOperator{\cbi}{CBI}
\DeclareMathOperator{\gwi}{GWI}
\DeclareMathOperator{\gw}{GW}
\newcommand{\bra}[1]{\ensuremath{\left[ #1\right] }}
\newcommand{\z}{\ensuremath{\mathbb{Z}}}
\newcommand{\na}{\ensuremath{\mathbb{N}}}
\newcommand{\set}[1]{\ensuremath{\left\{ #1\right\} }}
\DeclareMathOperator{\id}{Id} %
\newcommand{\indi}[1]{\si_{#1}}
\newcommand{\si}{{\ensuremath{\bf{1}}}}
\newcommand{\esp}[1]{\ensuremath{\se\! \left( #1 \right)}}

\newcommand{\clo}[1]{\ensuremath{\overline{#1}}}
\newcommand{\abs}[1]{\hspace{.25mm}\left|#1\right|\hspace{.25mm}}

\newcommand{\sag}[1]{\sigma\!\paren{#1}}
\newcommand{\ra}{\ensuremath{\mathbb{Q}}}
\newcommand{\proba}[1]{\ensuremath{\sip\! \left( #1 \right)}}
\newcommand{\sip}{\mathbb{P}}
\newcommand{\mc}[1]{\ensuremath{\mathscr{#1}}}

\newcommand{\floor}[1]{\ensuremath{\lfloor #1\rfloor}}

\newcommand{\eps}{\ensuremath{ \varepsilon}}
\newcommand{\fun}[3]{\ensuremath{#1:#2\to #3}}

\newcommand{\ceil}[1]{\ensuremath{\lceil #1 \rceil}}
\newcommand{\varc}[2]{\ensuremath{\imf{\text{Var}}{\cond{#1}{#2}}}}
\newcommand{\espc}[2]{\ensuremath{\imf{\se}{\cond{#1}{#2}}}}
\newcommand{\var}[1]{\ensuremath{\text{Var}\! \left( #1 \right)}}

\title{Invariance principles for local times in regenerative settings}
\author{Aleksandar Mijatovi{\'c}}
\address{ Department of Statistics, University of Warwick 
\\ \& The Alan Turing Institute 
}
\email{a.mijatovic@warwick.ac.uk}

\author{Ger\'onimo Uribe Bravo}
\address{Instituto de Matem\'aticas\\ Universidad Nacional Aut\'onoma de M\'exico
}
\email{geronimo@matem.unam.mx}


\thanks{AM supported by EPSRC grant EP/P003818/1 and the Turing Fellowship funded by the Programme on Data-Centric Engineering of Lloyd's Register Foundation; 
GUB's research supported by 
CoNaCyT
grant FC-2016-1946 and UNAM-DGAPA-PAPIIT grant no. IN115217. 
} 
\usepackage[dvipsnames]{xcolor}

\begin{document}
\begin{abstract}
Consider a 
stochastic 
process $\mathfrak{X}$, regenerative at a state $x$ which is instantaneous and regular. 
Let $L$ be a regenerative local time for $\mathfrak{X}$ at $x$. 
Suppose furthermore that $\mathfrak{X}$ can be approximated by discrete time regenerative processes $\mathfrak{X}^n$ 
for which $x$ is accesible. 
We give conditions on $\mathfrak{X}$ and $\mathfrak{X}^n$ so that the naturally defined local time of $\mathfrak{X}^n$ converges weakly to $L$. 
This limit theorem generalizes previous invariance principles that have appeared in the literature. 
\end{abstract}
\maketitle

\section{Introduction}
Counting the number of times a certain discrete-time stochastic process visits a given state 
is a simple operation on it. 
However, in continuous state-space, this quantity is typically zero, one or infinity, 
and so extending the notion is not a trivial task. 
The first example of such an extension is Paul L\'evy's construction of Brownian local time 
in Section 7.5 of \cite{MR0000919} or in  Th\'eor\`eme 47 of \cite{MR0190953}. 
This local time is a random non-decreasing continuous singular function, 
increasing only on the zero set of Brownian motion (whose Hausdorff dimension is $1/2$). 
Limit theorems then relating the discrete counting notion with the continuous one are not common, 
since local time is not a continuous functional of the sample path. 
Among the results available one has limit theorems
for the discrete local time at all states of random walks with finite variance in 
\cite{MR665738} and \cite{MR636771} 
or in the domain of attraction of a stable law in \cite{MR749918} 
(the case of continuous-time random walks is found in \cite{2019arXiv190700963K}).   
In higher dimensions, we mention the invariance principle of \cite{MR1169015} for local times of random walks on (neighborhoods of) curves in the Brownian domain of attraction. 
One also has limit theorems for the local time (stopped at two different stopping times) at all states 
of a sequence of spectrally positive compound Poisson process with negative drift 
approximating a spectrally positive L\'evy process  in  \cite{MR3320960}, 
as well as for the local time at zero of L\'evy processes reflected at their minimum 
in \cite{MR1954248} and \cite{MR2663630}. 
One of the motivations of the latter works was to study convergence of the height process 
(a local time functional of a L\'evy process path) 
useful to study limit theorems for Galton-Watson trees. 
One of the motivations of this work is to find a common framework 
where one can obtain results in  the above directions. 
Very related are limit theorems for occupation times of Markov processes 
(which include some for local times) 
in \cite{MR0084222} and \cite{MR1949295}; 
these have, in the random walk case, strong versions in \cite{MR3060348} and \cite{MR3813993} 
(for random walk bridges). 
However, these results have focused on (self-similar) scaling limits of the occupation times of only one process 
and not on weak limits of local times.
Regarding our assumptions, 
we have in mind the setting where we have a sequence of discrete-time Markov processes 
converging weakly to another one, and our aim is to deduce that local times also converge. 
The assumptions are similar to ones featuring in results about weak convergence 
of regenerative processes of \cite{MR3189081}.

In order to highlight the fragile nature of the convergence of local times, consider the following simple example. 
Let $U^1=\abs{S}$, 
where $S$ is a simple symmetric random walk. 
By Donsker's invariance principle,  
$U^n=(U^1_{\floor{nt}}/\sqrt{n})$ converges weakly to $\abs{B}$ where $B$ is a Brownian motion. 
Borodin and Perkins proved (in the above cited works) that the (natural) local time 
\begin{lesn}
L^n_{k/n}=\#\set{j\leq k: U^n_{j/n}=0}
\end{lesn}of $U^n$, scaled as $L^n/\sqrt{n}$, converges weakly as $n\to\infty$, to the local time of Brownian motion. 
Fix $p_n\in [0,1]$. 
Now consider the Markov chain $\tilde V^n$, 
having the same transition probabilities $\tilde p_{i,j}$ as $U^1$ at states $i\geq 2$, 
and let  $\tilde p_{1,0}=p_n=1-\tilde p_{1,2}$ and $\tilde p_{0,1}=1$. 
When $p_n\to 0$, the scaled Markov chain $ V^n$ (obtained from $\tilde V^n$) 
still converges to $\abs{B}$ while its local time converges to zero, even without the square root scaling.

\subsection{Statement of the main result}

The setting of our main theorem will feature a conti\-nuous-time \cadlag\ stochastic process $\mathfrak{X}$ 
whose state-space is Polish. 
Recall that a function from $[0,\infty)$ into the state-space is  \cadlag\ if it is right-continuous and admits left limits. 
Recall also that the set of \cadlag\ functions can be given a metric which turns it into a Polish space, called the Skorohod space (cf. \cite{MR1700749} or \cite{MR1876437}). 
It is with respect to this metric that we consider convergence in probability of random elements in the Skorohod space. 

Our setting assumes that $\mathfrak{X}$ is regenerative  at an instantaneous and regular point $x$ at which it starts. 
Being regenerative at a point means that, 
whenever $\mathfrak{X}_t=x$, 
if $d_t=\inf\set{s>t: \mathfrak{X}_s=x}$, 
then $\mathfrak{X}_{d_t+\cdot}$ has the same law as $\mathfrak{X}$ 
and is independent of the latter stopped at time $d_t$. 
The point being regular and instantaneous means that 
whenever $\mathfrak{X}_t=x$, for any $\eps>0$ there exist $t_1,t_2\in (t,t+\eps)$ such that 
$\mathfrak{X}_{t_1}\neq x$ and $\mathfrak{X}_{t_2}=x$. 
For such a process, we can define a local time at state $x$, denoted $L$, 
which has continuous and increasing trajectories that are supported on the set of times $\mathfrak{X}$ visits $x$. 
A thorough treatment can be found in \cite[Ch. IV]{MR1406564} or \cite[Ch. 22]{MR1876169}. 
Local time is determined up to a normalization constant. 
Its right-continuous inverse is a subordinator (cf. \cite[Ch. III]{MR1406564}): 
a process with non-decreasing trajectories with independent and stationary increments. 
Let $\mu$ be its L\'evy measure. 
Since $x$ is regular and instantaneous, $\imf{\mu}{(0,\eps)}=\infty$ for any $\eps>0$. 

We will also consider a sequence of discrete-time processes $(\mathfrak{X}^n,n\geq 1)$ that are regenerative and start at $x$ and are  
indexed by $\na/n$ where $\na=\set{0,1,2,\ldots}$. 
Then, the successive excursions of $\mathfrak{X}^n$ away from $x$ are independent and identically distributed%
; we call $\mu_n$ the common law of the excursion lengths. 
Let $L^n$ be the (natural) discrete local time at $x$ of $\mathfrak{X}^n$ defined for $k\in\na$ by
\begin{linenomath}\begin{equation}
\label{equationDefiningDiscreteLocalTime}
L^n_{k/n}=\#\set{j\leq k: \mathfrak{X}^n_{j/n}=x}.
\end{equation}\end{linenomath}Define $g_t$ and $d_t$ (resp. $g^n_t$ and $d^n_t$) 
to be the left and right endpoints of the excursion of $\mathfrak{X}$ 
(resp. $\mathfrak{X}^n$) that straddle time $t\geq 0$, more precisely
\begin{linenomath}
\begin{align}
\label{definitonOfgn}
g^n_t=\max\set{k/n\leq t: \mathfrak{X}^n_{k/n}=x}, 
\quad \quad\quad &d^n_t=\min\set{k/n> t: \mathfrak{X}^n_{k/n}=x} 
\\
\quad 
g_t=\sup\set{s\leq t: \mathfrak{X}_s=x},
\quad\quad\text{ and }\quad 
&d_t=\inf\set{s>t: \mathfrak{X}_s=x}.\nonumber
\end{align}\end{linenomath}


Our main theorem, which will be later illustrated by examples, is the following: 
\begin{theorem}
\label{mainLimitTheorem}
Let $\mathfrak{X}$ be a continuous time \cadlag\ 
process whose state space is Polish and regenerates at a state $x$ that is 
regular and instantaneous and on which $\mathfrak{X}$ starts. 
For any normalization of local time for $\mathfrak{X}$ at $x$,  let $\mu$ be the L\'evy measure of inverse local time. 
Let $(\mathfrak{X}^n)$ be a sequence of discrete time 
regenerative processes where each $\mathfrak{X}^n$ are indexed by $\na/n$ and start at $x$. 
Denote by $\mathfrak{X}^n$ also the piecewise constant \cadlag\ extension to $[0,\infty)$. 
Assume that $\mathfrak{X}^n\to \mathfrak{X}$ 
in probability as random elements of Skorohod space and that, for any fixed $t>0$, 
$g^n_t$  and $d^n_t$ converge in probability to $g_t$ and  $d_t$ respectively. 
For any $l>0$, define $S_{l}=(l,\infty)$, and 
\begin{linenomath}
\begin{equation}
\label{normalizationSequenceDefinition}
a_n=\imf{\mu}{S_l}/\imf{\mu_n}{S_l}, 
\end{equation}
\end{linenomath}whenever $l$ is not an atom of $\mu$. 
Then $(\mathfrak{X}^n,L^n/a_n)\to (\mathfrak{X},L)$ in probability as $n\to\infty$. 
\end{theorem}


The above theorem reduces convergence of local times to questions regarding hitting or last visit times. 
We shall give numerous examples where the hypotheses are verified. 
The proof techniques range from simple pathwise arguments 
to establishing convergence in law of these random times. 
Note that our version of the approximating local times is by counting. 
There are alternatives such as using overshoots across a level as in \cite{MR665738}, cf.~\cite[Thm~5]{2018arXiv180805010M}.

When considering excursion lengths $d^n_t-g^n_t$, note that these take the values $1/n,2/n,\ldots$, 
implying that $\mu_n$ does not charge $0$. 
Remark that $g^n_t$ and $d^n_t$ make sense for all $t>0$ and not only on $\na/n$. 
Finally, the local time $L^n$ starts at $1$ and not at zero%
. 
Also, note that the temporal scaling for local times is built into the definition 
because of the assumption that the time parameter of $\mathfrak{X}^n$ is in $\na/n$. 
One of the consequences of our main theorem is that the choice of the level $l$ in the definition of $a_n$ is irrelevant. 
We also state and prove this directly in Proposition \ref{rightAndLeftEndpointConvergenceProposition} below. 


Note that, for the example $V^n$ described in the second paragraph of this section,  
we have $g^n_t=0$ with probability tending to $1$, 
so that it does not converge to $g_t$ 
(which has the arcsine law by a celebrated theorem of L\'evy also found in \cite{MR0000919}). 
Hence, the assumptions of our theorem do not hold. 


The proof of the above theorem uses 
the concept of nested arrays introduced in \cite{MR579823} and basic aspects of subordinators found in Chapter III of \cite{MR1746300}. 
It will be given in Section~\ref{SectionOnProofOfMainLimitTheorem} below, 
after we state and prove some of its implications.

Note that Theorem \ref{mainLimitTheorem} concerns only the convergence of local time at a single level. 
If the theorem is applicable at all levels, 
as it features convergence in probability 
(and not only weak convergence),
one would obtain from it at
least the weak convergence of the scalled counting processes to the 
limiting local time field at any finite set of levels. 
To conclude the convergence of random fields,  
additional analysis of the continuity properties of the limiting local time field
and of tightness would be necessary. 
As a firts step in this direction, 
we have chosen to restrict ourselves to a single level so that we may only require the
regenerative property of the processes involved. 

\subsection{Applications of the main result}

In our applications 
of Theorem \ref{mainLimitTheorem} above (Theorems~\ref{generalMarkovianTheorem},~\ref{randomWalkLocalTimeLimitTheorem},~\ref{reflectedRWTheorem}~\&~\ref{TheoremOnConvergenceOfLocalTimesOfGWIProcesses}), 
we assume only the weak convergence of the $\mathfrak{X}^n$ in various settings, 
then apply Skorohod's theorem 
so that the convergence holds almost surely in an adequate probability space and argue, case by case, 
that in this same space, 
the random variables $g^n_t$ and $d^n_t$ also converge in probability to $g_t$ and $d_t$. 
This allows us to get weak convergence of the $\mathfrak{X}^n$ together with their scaled local times. 
If we instead assume convergence in probability of $\mathfrak{X}^n$, 
then we can also conclude the convergence in probability of the local times in all the applications. 
Our processes $\mathfrak{X}^n$ are initially indexed by $\na/n$;
they can be extended to $[0,\infty)$ by making them \cadlag\ and constant between the consecutive points of $\na/n$.

Our next result is an indication of the generality of Theorem \ref{mainLimitTheorem}. 
We chose the hypotheses to make the proof as simple as possible. 
\begin{theorem}
\label{generalMarkovianTheorem}
Let $Z$ be regenerative at $z\in \re$ and satisfy 
the time-change equation
\begin{linenomath}
\begin{equation}
\label{TCEwrtStableProcess}
Z_t=z+X_{\int_0^t \imf{f}{Z_s}\, ds}
+\int_0^t \imf{b}{Z_s}\, ds+Y_t, \quad \text{for all }t\in [0,\infty), 
\end{equation}
\end{linenomath}where $X$ is a spectrally positive L\'evy process of infinite variation, 
$Y$ is a subordinator and $f\geq 0$. 
Assume that $f$ is 
continuous and non-zero at $z$ and 
$b$ is bounded in a neighborhood of $z$. 
Then $Z$ is regular and instantaneous at $z$ and we denote by $L$ its local time at $z$.
Finally, assume that $Z$ is quasi-left continuous. 

Let $(Z^n)$ be
a sequence of discrete-time and discrete-space processes, where each $Z^n$
is defined on $\na/n$, takes values in $z+\z /b_n$ (for some $b_n>0$) 
and is 
downward skip-free and regenerative at $z$. 
If the sequence $(Z^n)$ converges weakly to $Z$ 
then, for the sequence of normalising constants $(a_n)$ in Equation~\eqref{normalizationSequenceDefinition}  
and local times $L^n$ in~\eqref{equationDefiningDiscreteLocalTime}, 
we have the convergence $(Z^n, L^n/a_n)\to (Z,L)$ in distribution as $n\to\infty$. 
%
\end{theorem}

The fact that $Z$ solves \eqref{TCEwrtStableProcess} implies that it is
\cadlag\ and non-explosive. We denote by $Z_{t-}$ the left limit of $Z$ at $t>0$ and $Z_{0-}=Z_0$.
A solution $Z$ to the SDE\begin{linenomath} \begin{equation}
\label{SDEwrtStableProcess}
Z_t=z+\int_0^t \imf{g}{Z_{s-}}\, d\tilde X_s+\int_0^t \imf{b}{Z_s}\, ds, \quad \text{for all }t\in [0,\infty), 
\end{equation}
\end{linenomath}is a special case of \eqref{TCEwrtStableProcess} 
when $\tilde X$ is a spectrally positive stable process with index $\alpha\in (1,2]$, 
$Y=0$
and $g\geq 0$ if $\alpha\in (1,2)$ 
(see Subsection \ref{SubsectionOnApplicationOfMainTheoremToMarkovProcesses} for the proof of this assertion).
If $g$ and $b$ are globally Lipschitz then, according to \cite[Ch. 5]{MR2020294}, 
the (pathwise) unique (strong) solution $Z$ to the SDE \eqref{SDEwrtStableProcess} is a non-explosive strong Markov process, hence regenerative at any state. 
If $g(z)\neq 0$ then $Z$ satisfies the hypotheses of Theorem \ref{generalMarkovianTheorem}. 
In particular, we can consider the case when $\tilde X$ is a Brownian motion, so that $Z$ is a diffusion process. 
No matter how we approximate this diffusion by skip-free to one side regenerative processes, 
we deduce the joint convergence of processes together with the scaled local times. 

Another application occurs when $f=1$, $b=0$ and $Y=0$ in the time-change equation \eqref{TCEwrtStableProcess}, 
so that $Z$ (which equals $X$!) is an infinite variation spectrally positive L\'evy process. 
We deduce convergence of local times for downward skip-free regenerative processes converging to $X$, 
and in particular for sequences of downward skip-free random walks converging to $X$. 
The forthcoming Theorem \ref{randomWalkLocalTimeLimitTheorem} explores the case of two-sided jumps, 
where the arguments are more difficult and we restrict ourselves to approximations by a scaled random walk and where $X$ is a stable L\'evy process. 

Even though time-change equation \eqref{TCEwrtStableProcess} has been less studied than SDE \eqref{SDEwrtStableProcess}, it has been analyzed when $f=\id$, $b$ is linear and $Y$ is independent of $X$
in \cite{MR3098685} and shown to have strong approximation properties. 
In this case, the process $Z$ is a continuous-state branching process with immigration (hence Feller). 
Necessary and sufficient conditions for non-explosion of $Z$ are given in \cite[Corollary 5]{MR3098685}; 
when these hold, 
we conclude the convergence of local times at $z>0$ 
whenever $Z$ is approximated by downward skip-free regenerative processes. 
The family of continuous-state branching processes with immigration is exactly that of 
scaling limits of Galton-Watson processes with immigration. 
The latter, however, are not downward skip-free. 
We will be considering the convergence of local times at level zero for GWI processes 
(related to extinctions in the population) 
in Theorem \ref{TheoremOnConvergenceOfLocalTimesOfGWIProcesses} below. 
Finally, note that one can impose global conditions on $f$ and $b$ such that Theorem~\ref{generalMarkovianTheorem} applies at all levels $z$. 

Our next application of Theorem \ref{mainLimitTheorem} concerns weak
convergence of local times at level zero for lattice random walks in the domain
of attraction of a stable L\'evy process.  This problem has been addressed for
local times of simple symmetric random walks in \cite{MR665738} using nonstandard analysis.
The case of finite variance random walks was considered in \cite{MR636771}. The case when the 
increment distribution has infinite variance but is in the 
domain of attraction of an $\alpha$-stable law 
was analysed in \cite{MR749918}. The proof uses a generalisation of a theorem of Getoor~\cite{MR0397897}
and crucially depends on the fact that the laws of the limiting process and the approximating sequence
have independent increments. 
Theorem~\ref{mainLimitTheorem} applied in this setting yields the following result. 

\begin{theorem}
\label{randomWalkLocalTimeLimitTheorem}
Let $X^1$ be a random walk started at zero whose jump distribution has span $1$, that is, $\se(e^{iu X^1_1})=1$ if and only if $u\in 2\pi\z$. 
Let $X$ be a stable L\'evy process with index $\alpha\in (1,2]$ and local time $L$ at zero.  
Asume the existence of constants $(b_n)$ such that $X^1_n/b_n$ converges weakly to $X_1$. 
Let $X^n_{k/n}=X^1_{k}/b_n$ and let $L^n$ be the (natural or discrete) local time of $X^n$ at zero defined in \eqref{equationDefiningDiscreteLocalTime}. 
Then, for the sequence $(a_n)$ defined in \eqref{normalizationSequenceDefinition}, we have $(X^n,L^n/a_n)\to (X,L)$ in distribution as $n\to\infty$. 
\end{theorem}

As mentioned above, 
if we assume instead of weak convergence that the sequence of random walks in Theorem~\ref{randomWalkLocalTimeLimitTheorem} converge in probability, 
then the joint convergence with the local time also holds in probability. 
The scaling constants $a_n$ are regularly varying 
as  they are given in \cite{MR749918} in terms of  the asymptotic inverse 
(obtained through the de Bruijn conjugate in \cite{MR898871}) 
of the tails of the distribution of the increment of $X^1$.
This description makes them them explicit in some special cases (e.g. finite variance) but clearly setting dependent. 
Our proof of Theorem~\ref{randomWalkLocalTimeLimitTheorem} relies on 
Theorem~\ref{mainLimitTheorem} and 
the classical local limit theorems for random walks in the domain of attraction of stable processes. 
In fact, pathwise methods do not seem to be applicable in this context 
because the limiting process has two-sided jumps.  

On the other hand, the above theorem can be applied at any level and the limiting process $X$ admits a bicontinuous family of local times. 
We deduce the at least the finite-dimensional convergence of the scaled two-parameter local time process of $X^1$ to the limiting one for $X$.

We are also interested in multidimensional applications, 
e.g. the non-homogeneous random walks of \cite{2018arXiv180107882G}, 
which can be recurrent in any ambient dimension. 
These walks are spatially inhomogeneous  Markov processes, 
with a non-Markovian radial component which is regenerative at zero. 
Also,~\cite{2018arXiv180107882G} contains an invariance principle for these random walks 
(to a (Markovian) process solving a highly singular SDE) 
with the radial component of the limit equal to a Bessel process. 
We believe that with the roadmap set up for the proof of Theorem \ref{randomWalkLocalTimeLimitTheorem} 
we will be able to prove convergence of the local times. 
What we are lacking at the moment of writing is a local limit theorem for the one-dimensional distributions 
and a construction of its bridge laws, which we leave open for future work. 

Our next application of Theorem \ref{mainLimitTheorem} is to reflected random walks. 
Let $X$ be a L\'evy process. 
Define the running minimum process of $X$ by the formula $\underline X_t=\inf_{s\leq t} X_s$. 
Consider the reflection of $X$ at its running minimum defined by $R=X-\underline X$. 
Recall from Proposition 1 in \cite[Ch. VI]{MR1406564} that $R$ is a Feller process, hence regenerative at any state. 
Let $X^n$ be a random walk indexed by $\na/n$ and define the running miminum $\underline X^n$ 
by  $\underline X^n_{j/n}=\min_{i\leq j}X^n_{i/n}$ 
as well as the reflected process $R^n$ by $R^n_{j/n}=X^n_{j/n}-\underline X^n_{j/n}$ for $j\in \na$. 
\begin{theorem}
\label{reflectedRWTheorem}
Assume that $X^n_1\to X_1$ in distribution as $n\to\infty$. 
If $0$ is regular for 
both half-lines $(-\infty,0)$ and $(0,\infty)$ for the L\'evy process $X$, 
then 
$0$ is instantaneous and  regular for the reflected process $R$. 
Let $L$ be a local time at zero for $R$ and $L^n$ be the corresponding discrete local time for $R^n$. 
Then, 
for the sequence $(a_n)$ of Equation~\eqref{normalizationSequenceDefinition} corresponding to the local times of $R$ and $R^n$, 
we have $(R^n,L^n/a_n)\to (R,L)$ in distribution as $n\to\infty$. 
\end{theorem}

The above theorem is essentially \cite[Thm~2 ]{MR2663630}, except that their scaling constant is given by
\[
\tilde a_n=e^{-\sum_{k=0}^\infty \frac{1}{k}\proba{X^n_{k/n}<0}e^{-k/n}}.
\]%
Recall that the constant $a_n$ is proportional to $1/\p(d_0^n>t)$ (where
$d^n_0$ is the endpoint of the excursion straddling $0$ of $R^n$).  By Theorem
4 in \cite{MR2831081}, the constant $\tilde a_n$ can interpreted as
$1/\p(d_0^n>T)$, 
where $T$ is independent of $R^n$ and exponentially distributed. 
Theorem~\ref{reflectedRWTheorem} follows by a simple path-wise argument from Theorem~\ref{mainLimitTheorem}, 
see Section~\ref{SubsectionOnConvergenceOfLocalTimesForReflectedRandomWalks} 
below.


Finally, we apply Theorem \ref{mainLimitTheorem} in the setting of Galton-Watson processes. 
We will only consider a very special offspring distribution, the geometric one, 
to make the proof as simple as possible. 
We will comment in Section \ref{SectionOnConcludingRemarks} on possible extensions. 

\begin{theorem}
  \label{TheoremOnConvergenceOfLocalTimesOfGWIProcesses}
  Let $Z^1$ 
  be a Galton-Waton process with immigration, started at zero, 
  and having reproduction and immigration laws both equal to a geometric distribution on $\na$ 
  with parameters $1/2$ and $p$. 
  Assume that $\delta=p/(1-p)\in (0,1)$.
  Define $Z^n_{j/n}=Z^1_{j}/n$, and let $L^n$ be its local time at zero as in equation
  \eqref{equationDefiningDiscreteLocalTime}. 
  If $Z$ is the unique solution to the SDE
  \begin{equation}
  \label{SDEofBesselType}
  	Z_t=\int_0^t \sqrt{2 Z_s}\, dB_s+\delta t, 
  \end{equation}
then $0$ is instantaneous and regular for $Z$.  
Let $L$ be a Markovian local time of $Z$ at zero.  
Then 
  there exists a constant $c>0$ such that $(Z^n,L^n/n)\to(Z,cL)$ in distribution as $n\to\infty$. 
\end{theorem}

Note that in Theorem~\ref{SDEofBesselType} we have identified the scaling constants for the local time explicitly.
This is feasible in this setting 
because we consider the specific weak approximation of the continuous-state brancing process
with immigration
$Z$ by the Galton-Watson processes with immigration.  This setting implies that the
zero sets of such processes are random cutouts (more precisely defined in
Subsection \ref{subsectionOnGWI}), a fact that can be leveraged to obtain a weak
limit theorem for $L^n/n$ as an intermediate step in the proof of 
Theorem~\ref{TheoremOnConvergenceOfLocalTimesOfGWIProcesses}.  

The stochastic process $(Z_{2 t})$ 
(which has the same law as $(2Z_{t})$) 
is actually a squared Bessel process of dimension $2\delta$. 
In the above theorem, 
we do not use the fact that the right-continuous inverse of the local time $L$ 
is a stable subordinator of index $1-\delta$, 
as originally proved for the squared Bessel process  in \cite{MR0247668}, 
but can actually deduce it from the proof of Theorem \ref{TheoremOnConvergenceOfLocalTimesOfGWIProcesses}. 
This also explains why we need the restriction on $\delta$:  
$0$ is non-polar for $Z$ (and, in fact, regular and instantaneous)  if and only if $\delta\in (0,1)$. 


\subsection{Organization of the paper}

There are two main ways of verifying that the hypotheses of Theorem \ref{mainLimitTheorem} hold: 
via pathwise arguments and via distributional methods. 
The former is based on elementary arguments and thus 
we first establish, in Section \ref{SectionOnPathwiseApplicationsOfMainTheorem}, 
the two direct applications of Theorem \ref{mainLimitTheorem}
given in Theorems \ref{generalMarkovianTheorem} and \ref{reflectedRWTheorem}. 
We then proceed to the proof of Theorem \ref{mainLimitTheorem} in Section 
\ref{SectionOnProofOfMainLimitTheorem}. 
We finally address, 
in Section \ref{SectionOnApplicationsViaDistributionalMethods},  
two additional applications of our main theorem, obtained via distributional methods and 
represented by Theorems \ref{randomWalkLocalTimeLimitTheorem} and 
\ref{TheoremOnConvergenceOfLocalTimesOfGWIProcesses}. 

\section{Convergence of local times of Markov processes}
\label{SectionOnPathwiseApplicationsOfMainTheorem}
The aim of this section is to prove Theorems \ref{generalMarkovianTheorem} and \ref{reflectedRWTheorem} 
on the convergence of local times for Markov processes and reflected random wallks. 

\subsection{Convergence of local times of Markovian solutions to stochastic equations}
\label{SubsectionOnApplicationOfMainTheoremToMarkovProcesses}
The proof of Theorem~\ref{generalMarkovianTheorem} is based on showing that  
the assumptions of Theorem~\ref{mainLimitTheorem} hold in this setting. 
First of all, we reduce the assertion about the solution $Z$ of SDE~\eqref{SDEwrtStableProcess},
stated immediately after Theorem~\ref{generalMarkovianTheorem},
to the one pertaining the time-change equation in~\eqref{TCEwrtStableProcess}. 
Recall that $g\geq 0$ if $\alpha\in (1,2)$. 
Then Knight and Kallenberg's theorems 
on stochastic integrals with respect to Brownian motion and stable L\'evy processes 
(see \cite{MR1158024} for the latter) 
imply the existence of $X$, 
which has the same law as $\tilde X$ in SDE~\eqref{SDEwrtStableProcess}, 
such that:
\[
Z_t=z+X_{\int_0^t \abs{\imf{g}{Z_s}}^\alpha\, ds}+\int_0^t \imf{b}{Z_s}\, ds. 
\](Kallenberg's theorem would need two independent stable processes 
if $g$ takes also negative values and $\alpha\in (1,2)$.) 
But then, $Z$ is a solution to the time-change equation \eqref{TCEwrtStableProcess} 
with $Y=0$ and $f=\abs{g}^\alpha$. 
Hence, it suffices to prove the stated result for $Z$ satisfying \eqref{TCEwrtStableProcess}. 
We start by establishing the following claim.

\noindent \textbf{Claim.} 
During any excursion of $Z$ away from $z$, 
the regenerative state $z$ is only approached at the endpoints of the excursion. 
Put differently 
for any $t>0$, 
such that $g_t<d_t$, we have  $Z_s\neq z$ and 
$Z_{s-}\neq z$ for any $s\in (g_t,d_t)$ almost surely. 

The former assertion $Z_s\neq z$ follows by the definition of excursion. 
We now establish the latter using quasi-left continuity.
Indeed, for any $t\geq 0$, let us define $T^n_t=\inf\set{s \in [t,d_t): \abs{Z_s-z}<1/n }$ 
(with the usual convention $\inf\emptyset = \infty$).  
Notice that $n\mapsto T^n_t$ is non-decreasing, 
so that we can define $T_t=\lim_{n\to\infty} T^n_t$. 
In particular, on the event $\{T_t<\infty\}$, we have $T_t\leq d_t$.
Thus, since left-limits of the paths of $Z$ exist, $z$ is only reached at endpoints of excursions if and only if  
for all rational $t>0$ 
we have 
$T_t= d_t$ 
on the event $\{T_t<\infty\}\cup\{d_t=\infty\}$.

To prove this assertion, fix $t\geq 0$. 
On the event that $\{T_t<\infty\}$, the assumed quasi-left continuity of $Z$ means that 
$\lim_{n\to\infty} Z_{T^n_t}=Z_{T_t}$ almost surely. 
We now prove that $Z_{T_t}=z$, which implies that $T_t=d_t$ on $\{T_t<\infty\}$. 
Assume that $\lim_{n\to\infty} Z_{T_t^n}\neq z$. 
We now examine the cases when $T_t^n=T_t$ for all sufficiently large $n$ and when $T_t^n<T_t$ for all $n$. 
In the first case, by definition of $T_t^n$, 
there exists a sequence $\eps_n\downarrow 0$ such that $\abs{Z_{T_t+\eps_n}-z}<1/n$, 
implying by the right continuity of the paths of $Z$ that $Z_{T_t}=z$, contradicting our assumption. 
In the second case the existence of the left limit at $T_t$ and  our assumption imply that 
$Z_{T_t-}=\lim_n Z_{T_t^n}\neq z$. 
On the other hand, by the definition of $T_t^n$, 
there exists a positive sequence $\eps_n\to 0$ such that $T_t^n+\eps_n<T_t$ 
and   $\abs{Z_{T_t^n+\eps_n}-z}<1/n$. 
The existence of left limits of the paths of $Z$ at $T_t$ implies that $Z_{T_t-}=\lim_n Z_{T_t^n}=z$, 
a contradiction. 
We have just proved that $\lim_{n\to\infty} Z_{T^n_t}=Z_{T_t}=z$ and so $T_t=d_t$ on $\{T_t<\infty\}$. 
To prove the claim it 
remains to show that  on 
$\{d_t=\infty\}$
we have $T_t=\infty$, which follows by analogous arguments. 


Consider a spectrally positive L\'evy process $X$ of infinite variation. 
Recall that the hypothesis of infinite variation means that either X features a Gaussian component 
or that its jumps (on any finite interval) are not summable.


\begin{description}
\item[Regular and instantaneous character of $z$]Since $X$ is of infinite variation, 
the main result in  \cite{MR0242261} tells us that 
$\limsup_{t\to 0} X_t/t=\infty$ and $\liminf_{t\to 0} X_t/t=-\infty$ almost surely. 
(See also \cite{2019arXiv190304745A} for a different proof.) 
We now prove that the same holds for $Z-z$. 
Let $t^+_n$ and $t^+_n$ be sequences decreasing to zero 
and such that $\lim_nX_{t^+_n}/t^+_n=\infty$ and $\lim_nX_{t^{-}_n}/t^{-}_n=-\infty$. 
(Note that the sequences depend on the path of $X$. ) 
Since $\imf{f}{z}\neq 0$ and $Z$ has \cadlag\ paths, 
we see that $C_t=\int_0^t \imf{f}{Z_s}\, ds$ is strictly increasing and continuous on a neighborhood of $t=0$. 
Using the notation $\sim$ for asymptotic equivalence  
(for two positive functions $g$ and $h$ we write $g\sim h$ if $g(t)/h(t)\to 1$ as $t\to 0$), 
we also see that $C_t\sim \imf{f}{z}t$ as $t\to 0$. 
Since $b$ is bounded and $Y$ is a subordinator, 
there exists a constant $M>0$ such that $\int_0^t\abs{\imf{b}{Z_s}}\, ds\leq Mt$ and $Y_t\leq Mt$ 
on a neighborhood of zero (use Proposition III.8 in \cite{MR1406564} for $Y$).  
We may assume that, for sufficiently large $n$,  
these inequalities hold for $t\in [0,t^+_n]$ and that there exists $s^+_n$ such that $C_{s^+_n}=t^+_n$. 
Since $\imf{f}{z}s^+_n\sim C_{s^+_n}=t^+_n$, we obtain
\[
\lim_{n\to\infty} \frac{Z_{s^+_n}-z}{s^+_n}
=\imf{f}{z}\lim_{n\to\infty} \frac{X_{t^+_n}+\int_0^{t^+_n} \imf{b}{Z_s}\, ds+Y_{t^+_n}}{t^+_n}
=\infty. 
\]Proceed similarly for $s^{-}_n$ to get $\lim_{n\to\infty} (Z_{s^{-}_n}-z)/s^{-}_n=-\infty$. 
In particular, since $Z$ has no downward jumps, 
we see that $0$ is regular for both half-lines $(-\infty,0)$ and $(0,\infty)$, as well as regular for itself. 
Hence, $0$ is regular and instantaneous. 
\item[Convergence of right endpoints of excursions] 
Consider a fixed $t>0$; recall that $Z$ is continuous at $t$ almost surely. 
Recall also that $Z$ only approaches $z$ at the endpoint of excursions away from $z$ almost surely. 
By the subsequence principle for convergence in probability (as in \cite[Thm. 20.5]{MR2893652}), 
we will assume that $Z^n\to Z$ almost surely rather than in probability, and prove that $d^n_t\to d_t$ almost surely. 
Fix an $\omega$  such that the trajectory  $Z(\omega)$  is continuous at $t$,
approaches $z$ only at the endpoints of any excursion (see Claim above) 
and $Z^n(\omega)\to Z(\omega)$ in Skorohod space.  
We use this $\omega$ throughout this paragraph but omit it for ease of notation.
We first prove that $\liminf d^n_t\geq d_t$. 
In the case $t=d_t$ this is immediate since $d^n_t\geq t=d_t$ for all $n$. 
In the case $t<d_t<\infty$, for any $\eps\in (0,d_t-t)$ such that $d_t-\eps$ is a
continuity point of $Z$, 
since
$Z$ approaches $z$ only at endpoints of excursions, we see that there
exists $\delta>0$ such that $\abs{Z_s-z}>\delta$ for all $s\in [t,d_t-\eps]$. 
%
%
Since $Z$ is continuous at $t$ and $d_t-\eps$, 
Lemma 1 in \cite[Ch 3\S16]{MR1700749} implies the existence of 
a sequence of increasing homeomorphisms $\fun{\lambda_n}{[t,d_t-\eps]}{[t,d_t-\eps]}$ such that 
\[
	\sup_{s\in[t,d_t-\eps]}\abs{Z^n_s-Z_{\lambda_n(s)}}\to 0
\]as $n\to\infty$. Hence,  $\abs{Z^n_s-z}>\delta/2$ for $s\in [t,d_t-\eps]$ and all sufficiently large $n$. 
We deduce that $d^n_t>d_t-\eps$. 
Since the continuity points of $Z$ are dense in $[0,\infty)$, 
$\eps\in (0,d_t-t)$ can be chosen to be arbitrarily small, implying
$\liminf_{n\to\infty} d^n_t\geq d_t$. 
If 
$d_t=\infty$, an analogous argument based on the application of 
Lemma 1 in \cite[Ch 3\S16]{MR1700749} on the interval $[t,M]$ (for an arbitrarily large continuity point $M$
of the process $Z$) yields 
$\liminf_{n\to\infty} d^n_t=\lim_{n\to\infty} d^n_t= \infty$. 

We now prove that $\limsup d^n_t\leq d_t$. 
Recall that $d_t$ is a stopping time and so, by regularity of both half-lines,
we see that for every $\eps>0$ there exist $d_t<t_1<t_2<d_t+\eps$ such that
$Z_{t_1}>z>Z_{t_2}$.  By the convergence in the Skorohod topology, there exist
two sequences $(t^n_1)$ and $(t^n_2)$ converging to $t_1$ and $t_2$, such that
$Z^n_{t^n_1}>z>Z^n_{t^n_2}$ for large enough $n$. 
The skip-free character of $Z^n$ now implies the existence of $r^n$ satisfying 
$d_t\leq t^n_1\leq r^n\leq t^n_2<d_t+\eps$ and such that $Z^n_{r^n}=z$ for all sufficiently large $n$. 
But then, since $t\leq r^n$ for all sufficiently large $n$, we have $d^n_t\leq d_t+\eps$. 
Since $\eps$ was arbitrary, we obtain that $\limsup_{n} d^n_t\leq d_t$. 
In conclusion, we see that $\lim_n d^n_t=d_t$. 
\item[Convergence of left endpoints of excursions] 
As in the previous item, 
we will assume that $Z^n\to Z$ almost surely rather than in probability, 
and prove that $g^n_t\to g_t$ almost surely. 
Fix $t>0$. 
If $g_t=t$ then $\limsup_{n} g^n_t\leq t=g_t$. 
If $g_t<t$, recall that almost surely, $\proba{g_t<t=d_t}=0$ (cf. proof of \cite[Lemma 1.11]{MR1746300}). 
Thus $t<d_t$ and hence, by the Claim above, 
for any $\eps\in (0,t-g_t)$ 
the process $Z$ does not approach $z$ on $[g_t+\eps, t]$.  
If $g_t+\eps$ is a continuity point of $Z$, then $Z^n$ is also bounded away from $z$  on $[g_t+\eps,t]$  
by  Lemma 1 in \cite[Ch 3\S16]{MR1700749}
(as in the convergence of right-endpoints). 
Put differently, $\limsup_n g^n_t\leq g_t+\eps$ for all small $\eps>0$ and hence, $\limsup_{n} g^n_t\leq g_t$. 

We now prove that $g_t\leq \liminf_n g^n_t$. 
As a preliminary, note that both half-lines are regular for $Z$ after each $d_q$ 
simultaneously for all rational positive $q$, almost surely.  
We first need to prove that, for any fixed $t>0$, almost surely, 
$Z$ visits both half-lines on any interval of the form $(g_t-\eps, g_t)$. 
Note that on the set $g_t<t$, $g_t$ is an accumulation point of $\mc{Z}=\{s\in[0,\infty):Z_s=z\}$ from the left, 
since $\clo{ \mc{Z}}$ is perfect. 
Next, for any fixed $t$, we have $t=g_t$ if and only if $t=g_t=d_t$ 
and then $t$ is also an accumulation point of $\mc{Z}$ from the left, 
as shown in the proof of  \cite[Lemma 1.11]{MR1746300}. 
Hence, there exists a sequence $d_{q_n}\uparrow g_t$ where $q_n$ is rational and $q_n\uparrow g_t$. 
But, immediately after each $d_{q_n}$, $Z$ visits both half-lines. 
Hence, for any $m$, we can find $g_t-1/m<t^1_m<t^2_m<g_t$ such that $Z_{t^1_m}>z$ and $Z_{t^2_m}<z$. 
But then, for any $m$, there exists $N$ such that for any  $n\geq N$ 
there exist $\tilde t^1_m$ and $\tilde t^2_m$ converging to $t^1_m$ and $t^2_m$ 
such that $Z^n_{\tilde t^1_m}>z$ and $Z^n_{\tilde t^2_m}<z$. 
By the skip-free character of $Z^n$, there exists a point in $[\tilde t^1_n,\tilde t^2_n]$ on which $Z^n$ equals $z$. 
Hence,  $g_t\leq \liminf_n g^n_t$. 
\end{description}

\begin{remark}
Note that the spectrally one-sided character of $X^n$ is important 
to ensure that to go from a value greater than $z$ to  value less than $z$ one is forced to pass through $z$; 
this allows one to use simple path-wise arguments. 
\end{remark}

\subsection{Convergence of local times of reflected random walks}
\label{SubsectionOnConvergenceOfLocalTimesForReflectedRandomWalks}

Our proof of Theorem \ref{reflectedRWTheorem} amounts to checking that 
the conditions of Theorem \ref{mainLimitTheorem} are satisfied. 
In this case, we can use the simple pathwise arguments as in the proof of Theorem 
\ref{generalMarkovianTheorem}. 

\begin{proof}[Proof of Theorem \ref{reflectedRWTheorem}]
Recall from Skorohod's Theorem that the weak convergence  $X^n_1\to X_1$ 
already entails the convergence of $X^n$ to $X$ in Skorohod space (cf. Theorem 16.14 in \cite{MR1876169}). 
We can then assume (again by a theorem of Skorohod) that convergence takes place almost surely. 
We also recall the following result from \cite[Propositions 2.2 and 2.4]{MR0433606}  
which holds under our assumption of regularity of both half-lines:
for each $t>0$, there exists $\rho_t\in (0,t)$ such that 
\[
	\set{s\in [0,t]: X_s=\underline X_s\text{ or }X_{s-}=\underline X_t}=\set{\rho_t}
\]and $X_{\rho_t}=X_{\rho_t-}=\underline X_{\rho_t}$. 
By considering the above simultaneously for all rational $t$, 
it follows that during an excursion of $R=X-\underline X$ above $0$, $R$ only approaches zero at the endpoints. 
\begin{description}
\item[Regular and instantaneous character of $0$]
Let $T_\eps=\inf\set{t\geq 0: X_t\leq -\eps}$. Note that $T_\eps>0$. Regularity of $(-\infty,0)$ tells us that $T_\eps\to 0$. Note that $X_{T_\eps}=\underline X_{T_\eps}$, even if  $X$ is discontinuous at $T_\eps$. 
Hence, $R_{T_\eps}=0$ for all $\eps>0$ and so $0$ is regular for $\set{0}$ for $R$. 
However, by regularity of $0$ for $(0,\infty)$, we see that $X_{T_\eps+\cdot}-X_{T_\eps}$ visits $(0,\infty)$ on any right neighborhood of $T_{\eps}$, which implies the same for $R$. 
Hence, $0$ is instantaneous for $R$. 
\item[Convergence of left endpoints of excursions] 
Note for any $X^n$,
\begin{lesn}
g^n_t=\sup\set{s\leq t: X^n_s=\underline X^n_{t}},
\end{lesn}the last time $X^n$ reaches its overall minimum on the interval $[0,t]$. Since, as  explained above, 
we have that $g_t=\rho_t$  and $\underline X_t=X_{\rho_t-}< X_s\wedge X_{s-}$ for any $s\neq \rho_t$, 
we see that $g_t^n\to g_t$. 
Indeed, for any $\eps>0$ such that $g_t-\eps$ and $g_t+\eps$ are continuity points of $X$, 
let $V=[0,g_t-\eps]\cup [g_t+\eps,t]$. 
Then, by the convergence $X^n\to X$ on $V$ and continuity of $X$ at $\rho_t$, 
we see that $\inf\set{X^n_s:s\in V}\to \inf\set{X_s:s\in V}>X_{\rho_t}=\lim_n X^n_{\rho_t}$. 
Therefore, the infimum of $X^n$ on $[0,t]$ is achieved in $[g_t-\eps,g_t+\eps]$ for $n$ large enough. 
Put differently, $g^n_t\in [g_t-\eps,g_t+\eps]$ for large enough $n$. 
Since $\eps$ can be chosen arbitrarily small, we deduce that $g^n_t\to g_t$. 
\item[Convergence of right endpoints of excursions] 
We first prove that $\liminf d^n_t\geq d_t$. 
In the case when $t=d_t$ this is immediate since $d^n_t\geq t=d_t$. 
In the case $t<d_t$,  for any  $\eps\in (0,d_t-t)$ such that $X$ is continuous at $d_t-\eps$, note that
\begin{lesn}
\lim_{n\to \infty}\underline X^n_t=\underline X_t<\underline X_{[t,d_t-\eps]}
=\lim_{n\to\infty } \underline X^n_{[t,d_t-\eps]}. 
\end{lesn}We deduce that $X^n$ cannot reach its minimum on $[0,t]$ 
on the interval $[t,d_t-\eps]$ for $n$ large enough, which says that $d^n_t\geq d_t-\eps$. 
Equivalently $\liminf d^n_t\geq d_t$. 
We now prove that $\limsup d^n_t\leq d_t$. 
Recall that $d_t$ is a stopping time and so, by regularity, 
we have that $\underline X_{d_t+\eps}<\underline X_{d_t}=\underline X_t$. 
If there exists $\eps>0$ such that $d^n_t>d_t+\eps$ for large $n$ then 
$\underline X^n_t=\underline X^n_{d^n_t}$, and so we get the contradiction
\begin{lesn}
\underline X_t
=\lim_{n\to\infty}\underline X^n_t
=\lim_{n\to\infty}\underline X^n_{d^n_t}\leq \lim_{n\to\infty}\underline X^n_{d_t+\eps}
= \underline X_{d_t+\eps}<\underline X_t. 
\end{lesn}We conclude  that $\limsup d^n_t\leq d_t$ and hence that $\lim d^n_t=d_t$ almost surely for any $t>0$. From the above, we get that almost surely and simultaneously for all positive rational numbers $q$, 
we have $g^n_q\to g_q$ and $d^n_q\to d_q$. 
\end{description}
We have therefore established the assumptions of Theorem \ref{mainLimitTheorem} which ends the proof of Theorem \ref{reflectedRWTheorem}. 
\end{proof}

\section{Proof of the local time invariance principle}
\label{SectionOnProofOfMainLimitTheorem}
We now present the proof of our main limit theorem stated as Theorem \ref{mainLimitTheorem}. 
The proof has 3  main components. 
First, Proposition \ref{rightAndLeftEndpointConvergenceProposition} tells us that, 
under the hypotheses of Theorem \ref{mainLimitTheorem}, 
the excursion lengths converge and that the excursion measure converges vaguely. 
This proposition explains why the scaling sequence is adequate to obtain a non-trivial limit. 
Second, Lemma \ref{largeJumpConvergenceLemma} tells us that if we erase small excursions 
(say below a certain fixed positive length), 
the resulting inverse local times converge to a subordinator (with no drift or small jumps) in probability.  
This is where the conclusion of convergence in probability is obtained for Theorem \ref{mainLimitTheorem}. 
It is based on the concept of nested arrays and the resulting  approximations of local time of \cite{MR579823}. 
Finally, the small jumps in inverse local time are controlled by a weak limit theorem 
in Lemma \ref{weakConvergenceOfInverseLocalTimeLemma}.

To prove Theorem \ref{mainLimitTheorem}, 
we first start with a result implying that conditioned excursions lengths converge. 
Recall that $\mu_n$ stands for the common law of the excursions of $\mathfrak{X}^n$ away from $x$, 
which coincides with the law of $d_n(0)-g_n(0)$  defined in Equation \eqref{definitonOfgn}. 
Similarly, $\mu$ stands for the L\'evy measure of the right-continuous inverse of the chosen local time $L$. 
\begin{proposition}
\label{rightAndLeftEndpointConvergenceProposition}
In the setting of Theorem \ref{mainLimitTheorem}, 
let $l>0$ be a continuity point of the intensity $\mu$ of excursion lengths. 
Then, the sequence of lengths of excursions of lengths greater than $l$ of $\mathfrak{X}^n$ 
converge to the corresponding ones for $\mathfrak{X}$ in probability 
and the laws $\imf{\mu_n}{\cond{\cdot}{S_{l}}}$ converge weakly to $\imf{\mu}{\cond{\cdot}{S_{l}}}$. 
Furthermore, the sequence of scaling constants $(a_n)$ satisfies 
$a_n\to\infty$ and $a_n\mu_n$ converges vaguely to $\mu$. 
\end{proposition}
\begin{proof}
We first note that there exists a subsequence $n_l$ such that 
$g^{n_l}_q$ and $d^{n_l}_q$  converge to $g_q$ and $d_q$ respectively for all rational $q$ almost surely. 
For notational simplicity, we focus on $g^n_q$.
Indeed, let $q_1,q_2,\ldots$ be an enumeration of the positive rational numbers 
and choose a first subsequence $n_{1,l}$ such that $g^{n_{1,l}}_{q_1}\to g_{q_1}$ almost surely as $l\to\infty$. 
Then, once we have a subsequence $n_{k,l}$ such that $g^{n_{k,l}}_{q_i}\to g_{q_i}$  almost surely 
for $i\leq k$ as $l\to\infty$, 
choose a further subsequence $n_{k+1,l}$ of $n_{k,l}$ such that  
$g^{n_{k+1},l}_{q_{k+1}}\to g_{q_{k+1}}$ as $l\to\infty$. 
Having done this for all $k$, define $n_l=n_{l,l}$. 
Note that $n_l$ is a subsequence of $n_{k,l}$ for all $k$. 
Hence, $g^{n_l}_{q_i}\to g_{q_i}$ for all $i$ almost surely. 

Let $(g,d)$ be the first excursion interval of $\mathfrak{X}$ of length $>l$ 
(where $l$ is as in the statement of Theorem \ref{mainLimitTheorem}) 
and $(g_n,d_n)$ the corresponding one for $\mathfrak{X}^n$. 
By the regenerative property, it suffices to prove that $d_n-g_n\to d-g$ in probability 
to conclude the convergence of the sequence 
(of excursion lengths greater than $l$). 
We will do this by a subsequence argument. 
If $n_l$ strictly increase to infinity, the above paragraph guarantees 
the existence of a further subsequence $n_{l_k}$ such that $g^{n_{l_k}}_{q}\to g_q$ and $d^{n_{l_k}}_q\to d_q$ 
for all rational $q$ almost surely. 
To simplify notation, assume that this holds for $g^n_q$ and $d^n_q$. 
We now prove that $g_n\to g$ and $d_n\to d$ almost surely. 
First note that if $d_n\to d$ then $g_n\to g$. Indeed, for any $q\in (d-l/2,d)\cap\ra$, 
we have the convergences $g^n_q\to g_q=g$ and $d^n_q\to d_q=d$. 
Since $d_n\to d$, then  the inequality $g_n\leq d_n-l<d-l/2<q<d_n$ holds for large $n$, 
implying that $g_n=g^n_q$ and so  $g_n=g^n_q\to g$. 
Next, assume that $d_n\not\to d$. 
Then for any $q\in (g,d)\cap \ra$, we have that $(g^n_q,d^n_q)\to (g,d)$ so that, 
for large $n$, $d^n_q-g^n_q>l$. 
Hence $d_n<g^n_q<d-l$ for large $n$. 
This upper bound allows us to find a subsequence $n_k$ so that $(g_{n_k},d_{n_k})$ 
converges to $(\tilde g,\tilde d)$ where $\tilde d-\tilde g\geq l$ and $\tilde d<d-l$. 
Let $\tilde q\in (\tilde g,\tilde d)\cap \ra$. 
Since  for large $k$, 
$(g_{n_k},d_{n_k})=(g^{n_k}_{\tilde q},d^{n_k}_{\tilde q})$ 
and $(g^{n_k}_{\tilde q},d^{n_k}_{\tilde q})\to (g_{\tilde q},d_{\tilde q})$, 
we see that $(\tilde g,\tilde d)=(g_{\tilde q},d_{\tilde q})$ is an excursion interval of length $\geq l$; 
since there are no excursion intervals of length exactly $l$, by hypothesis, 
then  $ \tilde d-\tilde g=d_{\tilde q}-g_{\tilde q}> l$. 
Since $\tilde d<d$, this contradicts the definition of $d$, and we conclude that $d_n\to d$. 
Hence, $d_n-g_n$ converges almost surely to $d-g$. 
Recall that we had simplified notation of the subsequence. 
Our conclusion is that the original sequence $d_n-g_n$ converges in probability to $d-g$. 

Since $\imf{\mu_n}{\cond{\cdot}{S_{l}}}$ is the law of $d_n-g_n$ (and analogously for $d-g$), 
we also deduce that the laws $\imf{\mu_n}{\cond{\cdot}{S_{l}}}$ converge weakly to $\imf{\mu}{\cond{\cdot}{S_{l}}}$. 
As a consequence, we  now prove the vague convergence of $a_n\mu_n$ to $\mu$. 
Let $\eps_i\downarrow 0$ be continuity points of $\mu$ with $\eps_i\leq l$. 
Then, 
\begin{linenomath}
\begin{align*}
a_n\imf{\mu_n}{
S_{\eps_i}
}
=
\frac{\imf{\mu}{S_l}}{\imf{\mu_n}{\cond{S_{l}}{S_{\eps_i}}}}
\to \imf{\mu}{S_{\eps_i}}.
\end{align*}\end{linenomath}In particular, we see that the definition of $a_n$ does not depend on $l$ asymptotically:
\[
a_n\sim \frac{\imf{\mu}{S_{\eps_i}}}{\imf{\mu_n}{S_{\eps_i}}}. 
\]Now, let $x>0$ be such that $\imf{\mu}{\set{x}}=0$. 
If $i$ is such that $\eps_i<x$, then 
\begin{linenomath}
\begin{align}
\label{vagueConvergence}
a_n\imf{\mu_n}{S_x}
\sim 
\imf{\mu}{S_{\eps_i}}\imf{\mu_n}{\cond{S_x}{S_{\eps_i}}}
\xrightarrow[n\to\infty]{} \imf{\mu}{S_{\eps_i}}\imf{\mu}{\cond{S_x}{S_{\eps_i}}} 
=\imf{\mu}{S_x}. 
%
\end{align}\end{linenomath}Hence, $a_n\mu_n$ converges vaguely to $\mu$ on $(0,\infty]$. 

We now prove that $a_n\to\infty$ using the assumption that $\mu$ is an infinite measure. 
The latter implies that $\imf{\mu}{S_{\eps_i}}\to\infty$  as $i\to\infty$. 
Since $a_n\sim \imf{\mu}{S_{\eps_i}}/\imf{\mu_n}{S_{\eps_i}}$ and $\imf{\mu_n}{S_{\eps_i}}\leq 1$, 
we see that $\liminf a_n\geq \imf{\mu}{S_{\eps_i}}$ for any $i$. 
%
\end{proof}

For the proof of Theorem \ref{mainLimitTheorem}, 
we will also need the following lemma, 
regarding the convergence in probability of the cummulative sums of big excursion lengths. 
Within this setting, recall that inverse local time, 
denoted $\tau$, is the right-continuous inverse of $L$ defined by
\[
	\tau_l=\inf\set{t\geq 0: L_t>l}. 
\]We also define an analogous inverse local time for the discrete time approximations $\tau^n$ given by
\[
	\tau^n_l=\inf\set{t\geq 0:L^n_t>l},
\]where $L^n$ has been extended by constancy to each of the intervals $[k,k+1)/n$. 
Since $L^n$ is integer valued, it follows that $\tau^n$ only jumps at integer times. 
Its $m$-th jump size, $\Delta \tau^n_{m}$, corresponds to the length of the $m$-th excursion away from $x$ of $\mathfrak{X}^n$; 
$\tau^n$ is then easily seen to be a random walk with values on $\na/n$. 
Similarly, we can define $\tau^{n,>a}$ as the random walk whose $n$-th jump 
is $\Delta \tau^{n}_m\indi{\Delta \tau^{n}_m>a}$. 
\begin{lemma}
\label{largeJumpConvergenceLemma}
Under the assumptions of Theorem \ref{mainLimitTheorem}, 
let $\tau$ be the inverse local time of $\mathfrak{X}$, \begin{lesn}
\imf{\Delta \tau}{t}=\imf{\tau}{t}-\imf{\tau}{t-}\quad\text{and}\quad\imf{\tau^{>a}}{t}=\sum_{s\leq t} \imf{\Delta \tau}{t}\indi{\imf{\Delta \tau}{s}>a}, 
\end{lesn}and similarly for $\tau^{n,>a}$. 
If the L\'evy measure of $\tau$ does not charge $a$ and is infinite, 
then $\imf{\tau^{n,>a}}{a_n\cdot}$ converges in probability (as a random element of Skorohod space) to $\tau^{>a}$. 
\end{lemma}

\begin{proof}
Note that we have already proved that the jump sizes of $\tau^{n,>a}$ 
converge in probability to those of $\tau^{>a}$ in Proposition \ref{rightAndLeftEndpointConvergenceProposition}. 
Let $T_{n,m}$ and $T_m$ be the times of the $m$-th jumps  of $\tau^n$ and $\tau$ of size $>a$. 
To prove the convergence of $\tau^{n,>a}$ to $\tau^{>a}$ in probability, 
it is therefore sufficient to prove that, for every fixed $m$,  
\begin{equation}
	\label{EqConvInProbTnmtoTn}
	\frac{T_{n,m}}{a_n}\to T_m
\end{equation}in probability as $n\to\infty$. 

Since $x$ is regular and instantaneous, the measure $\mu$ is infinite. 
Let $\eps_i$ be a sequence of continuity points of $\mu$ converging to zero. 
Let $\sigma_{i,m,a}$ be the number of jumps of $\tau$ of length $>\eps_i$ 
before the $m$-th jump of length $>a$ (and define $\sigma^n_{i,m,a}$ analogously). 
From Theorem 2.2 in \cite{MR579823} we see that, almost surely, 
\begin{lesn}
T_m=\lim_{i\to\infty}\frac{\sigma_{i,m,a}}{\imf{\mu}{S_{\eps_i}}}. 
\end{lesn}Let $\eps>0$ and deduce that
\begin{lesn}
\lim_{i\to\infty}\proba{\abs{T_m-\frac{\sigma_{i,m,a}}{\imf{\mu}{S_{\eps_i}}}}>\eps}=0.
\end{lesn}Now, decompose as follows
\begin{linenomath}
\begin{align}
\label{probabilityInequalityForConvInProbabilityOfTimesOfBigJumps}
\proba{\abs{\frac{\sigma_{i,m,a}}{\imf{\mu}{S_{\eps_i}}} -\frac{T_{n,m}}{a_n}}>\eps}
&\leq \proba{\abs{\frac{\sigma_{i,m,a}}{\imf{\mu}{S_{\eps_i}}} -\frac{\sigma^n_{i,m,a}}{\imf{\mu}{S_{\eps_i}}}}>\eps/2}
\\&+\proba{\abs{\frac{\sigma^n_{i,m,a}}{\imf{\mu}{S_{\eps_i}}}-\frac{T_{n,m}}{a_n}}>\eps/2}
\nonumber
\end{align}The convergence in \eqref{EqConvInProbTnmtoTn} follows 
if both probabilities on the right-hand side of \eqref{probabilityInequalityForConvInProbabilityOfTimesOfBigJumps} 
 tend to $0$ as $n\to\infty$ and then $i\to\infty$. 
For the first term, note that $\sigma_{i,m,a}^n$ converges in probability to $\sigma_{i,m,a}$ 
by our hypothesis (through Proposition \ref{rightAndLeftEndpointConvergenceProposition}), 
since they are equal with probability tending to one as $n\to\infty$ for any fixed $i$. 
For the last term, by Markov's inequality, it is sufficient to prove that
\begin{equation}
\label{secondMomentGoesToZeroEquation}
\lim_{i\to\infty}\limsup_{n\to\infty} \esp{\bra{\frac{\sigma^n_{i,m,a}}{\imf{\mu}{S_{\eps_i}}}-\frac{T_{n,m}}{a_n}}^2}=0.
\end{equation}

The law of $T_{n,m}$ is negative binomial with parameters $\imf{\mu_n}{S_a}$ and $m$ . 
Since $T_m$ is the sum of $m$ independent exponentials of parameter $1/\imf{\mu}{S_a}$, 
we get
\[\esp{T_{n,m}/a_n}\to m/\imf{\mu}{S_a}=\esp{T_m}
\]by the vague convergence of $a_n\mu_n$ of Proposition \ref{rightAndLeftEndpointConvergenceProposition}. 
Note that the conditional law of $\sigma^n_{i,m,a}-m$ given $T_{n,m}$ is binomial 
of parameters $T_{n,m}-m$ and $\imf{\mu_n}{\cond{S_{\eps_i}}{ S_a^c}}$. 
As we proved in Proposition \ref{rightAndLeftEndpointConvergenceProposition}, 
$\imf{\mu_n}{S_a}\to 0$ (or equivalently $a_n\to \infty$), 
which implies that $\imf{\mu_n}{S_a^c}\to 1$. 
Then, by the tower property, 
\begin{linenomath}
\begin{align*}
\esp{\frac{\sigma^n_{i,m,a}}{\imf{\mu}{S_{\eps_i}}}}
&=
\frac{m}{\imf{\mu}{S_{\eps_i}}}
+\esp{\frac{T_{n,m}-m}{\imf{\mu}{S_{\eps_i}}}}\frac{1}{a_n}\frac{a_n\imf{\mu_n}{S_{\eps_i}\setminus S_a}}{\imf{\mu_n}{S_a^c}}
\\&\xrightarrow[n\to\infty]{} \frac{m}{\imf{\mu}{S_{\eps_i}}}+\frac{m}{\imf{\mu}{S_a}\imf{\mu}{S_{\eps_i}}}\imf{\mu}{S_{\eps_i}\setminus S_a}
\\&=\frac{m}{\imf{\mu}{S_a}}.
\end{align*}
\end{linenomath}We deduce that
\[
\esp{\frac{\sigma^n_{i,m,a}}{\imf{\mu}{S_{\eps_i}}}-\frac{T_{n,m}}{a_n}}\to 0. 
\]Hence, in order to prove \eqref{secondMomentGoesToZeroEquation}, we only need to see that
\[
\lim_{i\to\infty}\limsup_{n\to\infty}\var{\frac{\sigma^n_{i,m,a}}{\imf{\mu}{S_{\eps_i}}}-\frac{T_{n,m}}{a_n}}=0. 
\]This is now done through the law of total variance. 

First, again using the law of $\sigma^n_{i,m,a}$ given $T_{n,m}$, we get
\begin{align*}
\varc{\frac{\sigma^n_{i,m,a}}{\imf{\mu}{S_{\eps_i}}}-\frac{T_{n,m}}{a_n}}{T_{n,m}}
&=
\varc{\frac{\sigma^n_{i,m,a}}{\imf{\mu}{S_{\eps_i}}}}{T_{n,m}}
\\
&=\frac{1}{\imf{\mu}{S_{\eps_i}}^2}
     \bra{T_{n,m}-m} 
     	\frac{\imf{\mu_n}{S_{\eps_i}\setminus S_a}}{\imf{\mu_n}{S_a^c}} 
     	\paren{1-\frac{\imf{\mu_n}{S_{\eps_i}\setminus S_a}}{\imf{\mu_n}{S_a^c}}},
\end{align*}so that, as $n\to\infty$, 
as $n\to\infty$, the vague convergence $a_n\mu_n\to \mu$ of Proposition \ref{rightAndLeftEndpointConvergenceProposition} 
implies that\begin{align}
\label{convergenceOfConditionalVariance}
\esp{\varc{\frac{\sigma^n_{i,m,a}}{\imf{\mu}{S_{\eps_i}}}
}{T_{n,m}}}
&=\frac{1}{\imf{\mu}{S_{\eps_i}}^2}\frac{m}{\imf{\mu_n}{S_a}}
      \frac{\imf{\mu_n}{S_{\eps_i}\setminus S_a}}{\imf{\mu_n}{S_a^c}} 
     	\paren{1-\frac{\imf{\mu_n}{S_{\eps_i}\setminus S_a}}{\imf{\mu_n}{S_a^c}}}
	\nonumber
\\&\xrightarrow[n\to\infty]{} \frac{m}{\imf{\mu}{S_{\eps_i}}^2}\bra{\frac{\imf{\mu}{S_{\eps_i}}}{\imf{\mu}{S_a}}-1}. 
\end{align}

On the other hand,
\begin{align*}
\espc{\frac{\sigma^n_{i,m,a}}{\imf{\mu}{S_{\eps_i}}}-\frac{T_{n,m}}{a_n}}{T_{n,m}}
&=\frac{1}{\imf{\mu}{S_{\eps_i}}}\bra{m+\paren{T_{n,m}-m}\frac{\imf{\mu_n}{S_{\eps_i}\setminus S_a}}{\imf{\mu_n}{S_a^c}}}-\frac{1}{a_n}T_{n,m}
\\&=\frac{T_{n,m}}{a_n}\bra{\frac{a_n\imf{\mu_n}{S_{\eps_i}\setminus S_a}}{\imf{\mu}{S_{\eps_i}}\imf{\mu_n}{S_a^c}}-1}+\beta
\end{align*}for some constant $\beta\in\re$. 
Hence, as $n\to\infty$, the vague convergence $a_n\mu_n\to \mu$ of Proposition \ref{rightAndLeftEndpointConvergenceProposition} gives
\begin{align}
\label{convergenceOfVarianceOfConditionalExpectation}
\var{\espc{\frac{\sigma^n_{i,m,a}}{\imf{\mu}{S_{\eps_i}}}-\frac{T_{n,m}}{a_n}}{T_{n,m}}}
&=\frac{1}{a_n^2}\bra{\frac{a_n\imf{\mu_n}{S_{\eps_i}\setminus S_a}}{\imf{\mu}{S_{\eps_i}}\imf{\mu_n}{S_a^c}}-1}^2\frac{m(1-\imf{\mu_n}{S_a})}{\imf{\mu_n}{S_a}^2}\nonumber
\\&\xrightarrow[n\to\infty]{} \frac{1}{\imf{\mu}{S_a}^2}\bra{\frac{\imf{\mu}{S_a}}{\imf{\mu}{S_{\eps_i}}}}^2m=\frac{m}{\imf{\mu}{S_{\eps_i}}^2}. 
\end{align}
\end{linenomath}By the convergences in \eqref{convergenceOfConditionalVariance} and \eqref{convergenceOfVarianceOfConditionalExpectation}, we see that\[
	\lim_{n\to\infty}\var{\frac{\sigma^n_{i,m,a}}{\imf{\mu}{S_{\eps_i}}}-\frac{T_{n,m}}{a_n}}
	= \frac{m}{\imf{\mu}{S_{\eps_i}}}. 
\]Since $\imf{\mu}{S_{\eps_i}}\to 0$ as $i\to\infty$, we have proved \eqref{secondMomentGoesToZeroEquation}. 
\end{proof}

We now need the convergence, this time only in law, of inverse local time. 
\begin{lemma}
\label{weakConvergenceOfInverseLocalTimeLemma} 
Under the assumptions of Theorem \ref{mainLimitTheorem}, let $\tau$ be the inverse local time of $\mathfrak{X}$ and similarly for $\tau^n$. 
Then $\tau^n(a_n\cdot )$ converges weakly to $\tau$. 
\end{lemma}
The proof will need basic aspects of subordinators found in Chapter 1 of \cite{MR1746300}. 
In particular, if $T$ is an exponential random variable of parameter $\theta$ independent of $\tau$ and $\Phi$ is the Laplace exponent of $\tau$, we will use the following formula which links the Laplace transform of $g_T$ and $d_T$ to $\Phi$:
\begin{linenomath}
\begin{equation}
\label{leftEndpointSubLTEq}
\esp{e^{-\alpha g_T-\beta d_T}}=\frac{\imf{\Phi}{\beta+\theta}-\imf{\Phi}{\beta}}{\imf{\Phi}{\alpha+\beta+\theta}}
. 
\end{equation}\end{linenomath}This formula is very similar to its discrete counterpart: 
\begin{linenomath}
\begin{equation}
\label{LaplaceTransformLeftEndpointOfExcursionDiscreteTime1}
\esp{e^{-\alpha g^n_T-\beta d^n_T}}
=\frac{\esp{e^{-\beta\tau^n_1}}-\esp{e^{-[\beta+\theta]\tau^n_1}}}{1-\esp{e^{-[\alpha+\beta+\theta]\tau^n_1}}}
\end{equation}
\end{linenomath}
The proof of \eqref{LaplaceTransformLeftEndpointOfExcursionDiscreteTime1} 
follows the same line of argument as that of \eqref{leftEndpointSubLTEq} found in \cite{MR1746300}, 
except that it is technically much simpler as it needs no discussion of resolvent densities. 
We now present it. 
Note that we have formula
\begin{lesn}
g^n_t=\max\set{\tau^n_k: \tau^n_k\leq t}. 
\end{lesn}Let $u_n$ be the discrete renewal density of the inverse local time $\tau^n$ given explicitly by\begin{lesn}
\imf{u_n}{k/n}=\proba{\mathfrak{X}^n_{k/n}=x}
\end{lesn}and let $U_n$ be the associated renewal function given by $\imf{U_n}{k/n}=\imf{u_n}{1/n}+\cdots +\imf{u_n}{k/n}$. 
Let $f$ be the generating function of $\tau^n_1$, so that$\imf{f}{t}=\esp{t^{\tau^n_1}}$, and 
define $T_n=\ceil{T n}/n$. 
Note that $T_n$ is a geometric random variable (on $\z_+/n$) with success probability $1-e^{-\theta/n}$. 
As defined,  $g^n_T=g^n_{T_n}$ equals the left endpoint of the excursion of $\mathfrak{X}^n$ straddling the random time $T_n$ (or $T$ for that matter). 
Note first that
\begin{linenomath}
\begin{align*}
\proba{g^n_{k/n}=i/n, d^n_{k/n}=j/n}
=\sum_{l=0}^\infty\proba{\tau^n_l=i/n, \tau^n_{ l+1}=j/n}
=u_n(i/n)\proba{ \tau^n_1=(j-i)/n}
\end{align*}\end{linenomath}for $i\leq k<j$. 
Hence
\begin{linenomath}
\begin{align*}
\esp{e^{-\alpha g^n_{T}-\beta d^n_T}}
&=\sum_{0\leq i\leq k<j} e^{-\alpha i/n-\beta j/n} u_n(i/n) \proba{\tau^n_1=(j-i)/n} e^{-\theta k/n}(1-e^{-\theta /n})
\\&=\sum_{i\geq 0} e^{-[\alpha+\beta+\theta] i/n} u_n(i/n)  \esp{e^{-\beta \tau^n_1}[1-e^{-\theta \tau^n_1}]}
\end{align*}\end{linenomath}On the other hand, 
\begin{linenomath}
\begin{align*}
\sum_{i\geq 0} e^{-\alpha i/n} u_n(i/n)
=\sum_i\sum_l e^{-\alpha i/n} \proba{\tau^n_l=i/n}
=\sum_l \esp{e^{-\alpha \tau^n_l}}
=\frac{1}{1-\esp{e^{-\alpha \tau^n_1}}}
\end{align*}\end{linenomath}Collecting terms, we get
\begin{linenomath}
\begin{align*}
\esp{e^{-\alpha g^n_{T}-\beta d^n_T}}=
\frac{\esp{e^{-\beta \tau^n_1}}-\esp{e^{-[\beta+\theta] \tau^n_1}]}}{1-\esp{e^{-[\alpha+\beta+\theta]\tau^n_1}}} , 
\end{align*}
\end{linenomath}which proves \eqref{LaplaceTransformLeftEndpointOfExcursionDiscreteTime1}. 

\begin{proof}[Proof of Lemma \ref{weakConvergenceOfInverseLocalTimeLemma}]
By hypothesis\begin{lesn}
\esp{e^{-\lambda g^n_t}}\to \esp{e^{-\lambda g_t}}
\end{lesn}as $n\to\infty$ for every $\lambda>0$. Then
\begin{linenomath}
\begin{align*}
\abs{\esp{e^{-\lambda g^n_{T_n}}}-\esp{e^{-\lambda g^n_{T}}}}
\leq \int_0^\infty \abs{\esp{e^{-\lambda g^n_t}}-\esp{e^{-\lambda g_t}}} \theta e^{-\theta t}\, dt
\to 0
\end{align*}\end{linenomath}by the dominated convergence theorem since the Laplace transforms are convergent for each fixed $t$ and uniformly bounded. 
Hence, the Laplace transform of $g^n_T$ converges pointwise to that of $g_T$ as $n\to\infty$. 

By Prohorov's theorem, 
there exists a subsequence $n_i$ such that $(\tau^{n_i}_{a_{n_i}},i\geq 1)$ converges weakly on $[0,\infty]$, 
say to a variable $\tilde \tau_1$, 
so that $\se(e^{-\lambda \tau^{n_i}_{a_{n_i}} })\to \se(e^{-\lambda\tilde\tau_1})$. 
We assert that $\tilde\tau_1$ is not almost surely equal to $\infty$. 
Indeed, assuming otherwise, since $\tau^{n_i,>a}_{a_{n_i}}$ converges weakly to a random variable which is finite with positive probability, by Lemma \ref{largeJumpConvergenceLemma}, 
we see that $\tau^{n_i,\leq a}_{a_{n_i}t}$ would have to converge to $\infty$ almost surely for all $t>0$. 
Hence, $\tau^{n_i}_{a_{n_i}\cdot}$ crosses $t$ by a jump of size $\leq a$ with probability tending to one for all $a>0$. 
We conclude that $\sip(g^{n_i}_t\in [t-a,t])\to 1$ as $i\to\infty$, so that $g^{n_i}_t\to t$ weakly. 
This contradicts the fact that $g^{n_i}_t\to g_t$ since, by our assumptions, $g_t\neq t$ has positive probability. 
Indeed, by Propositions 1 and 2 in \cite[Ch. III]{MR1406564}, we see that $\sip(g_t<t)\geq \imf{U}{t}\imf{\overline\pi}{t}$ where $U$ is the renewal measure associated to $\tau$, which satisfies $\imf{U}{t}>0$ for all $t>0$. We then need to choose $t$ such that $\imf{\overline\pi}{t}>0$, which is possible since $\tau$ is not a drift. 
Since $\tau^n$ is a random walk, we see that
\begin{lesn}
\esp{e^{-\lambda \tau^{n_i}_{ta_{n_i}} }}
=\esp{e^{-\lambda \tau^{n_i}_{a_{n_i}} }}^{\floor{a_nt}/a_n}
\to \esp{e^{-\lambda\tilde\tau_1}}^t. 
\end{lesn}Hence, the law of $\tilde\tau_1$ is infinitely divisible on $[0,\infty]$. Let $\tilde\Phi$ be its Laplace exponent and let $\tilde \tau$ be a subordinator with this exponent; recall that $\tilde\Phi$ is not identically zero since $\tilde\tau_1$ is not almost surely infinite. 
Also, we deduce the convergence of one-dimensional distributions of $(\tau^{n_i}_{a_{n_i}\cdot})$ to $\tilde  \tau$. 
As in the random walk case, the fdd convergence now follows. 
We assert that $\tau^{n_i}_{a_{n_i}\cdot}$ converges to $\tilde \tau$ on the Skorohod space of functions with values on $[0,\infty]$, 
where we have placed the metric $d(x,y)=\abs{e^{-x}-e^{-y}}$. 
For this, it is now only necessary to prove tightness which can be done simply through 
 Aldous's criterion. 
 Indeed, let $S_i$ be any stopping time for $\tau^{n_i}$ by some constant $a$ and let $h_i\to 0$. 
 We now get\begin{linenomath}\begin{align*}
 \esp{d(\tau^{n_i}_{a_{n_i}[S_i+h_i]},\tau^{n_i}_{a_{n_i}S_i})}
& =\esp{\abs{e^{-\tau^{n_i}_{a_{n_i}[S_i+h_i]}}-e^{-\tau^{n_i}{ a_{n_i}S_i}}}}
\\& =\esp{\abs{e^{-\tau^{n_i}_{a_{n_i}[S_i+h_i]} }-e^{-\tau^{n_i}_{a_{n_i}S_i}}}\indi{\tau^{n_i}_{S_i}<\infty}}
\\& \leq \esp{1-e^{-\lambda \tau^{n_i}_{a_n}}}^{h_n}\to 0,
 \end{align*}\end{linenomath}where the convergence follows since $\esp{e^{-\lambda \tau^{n_i}_{a_n}}}>0$. 

Let $\tilde g_T=\sup\set{\tilde\tau_t: \tilde\tau_t\leq T}$.  
We now assert that 
$g^{n_i}_T$ converges weakly to  $\tilde g_T$. 
First, we prove that $\tau^{n_i}_1\to 0$ in probability. 
Indeed, otherwise there would exist $\eps>0$ such that $\liminf \proba{\tau^{n_i}_1\geq \eps}\geq \eps$. 
But then $\proba{g^{n_i}_\eps=0}\geq \eps$ and by our assumed weak convergence, 
$\proba{g_\eps=0}\geq \eps$. 
On the other hand, since $a$ is regular for $\mathfrak{X}$, we see that $\proba{g_\eps=0}=0$. 
Having now proved that $\tau^{n_i}_1\to 0$ in probability, we obtain
\begin{lesn}
\esp{e^{-\lambda \tau^{n_i}_1}}\to1\quad \text{and}\quad -\log \esp{e^{-\lambda \tau^{n_i}_1}}\sim 1-\esp{e^{-\lambda\tau^{n_i}_1}}
\end{lesn}as $i\to\infty$. 
Equations \eqref{leftEndpointSubLTEq}, \eqref{LaplaceTransformLeftEndpointOfExcursionDiscreteTime1} and the convergence of $\tau^{n_i}$ to $\tilde \tau$ now give:
\begin{lesn}
\esp{e^{-\lambda g^{n_i}_T}}=
\frac{1-\esp{e^{-\theta \tau^{n_i}_1}}}{1-\esp{e^{-[\lambda+\theta]\tau^{n_i}_1}}}
\sim \frac{a_{n_i}\log \esp{e^{-\theta \tau^{n_i}_1}}  }{a_{n_i}\log \esp{e^{-[\lambda+\theta ]\tau^{n_i}_1}}}
\to \frac{\imf{\tilde \Phi}{\lambda}}{\imf{\tilde\Phi}{\lambda+\theta}}
=\esp{e^{-\lambda \tilde g_T}}. 
\end{lesn}We deduce that $\tilde g_T$ and $g_T$ have the same law, so that $c\Phi=\tilde \Phi$. 
Hence, the L\'evy measure of $\tilde \tau$ is $c\mu$. 
Let $a$ be a continuity point of $\mu$.  
Let $\chi$ be the first jump of $\tilde \tau$ strictly exceeding $a>0$ and let $\chi_i$ be 
the corresponding first jump for $\tau^{n_i}$. 
Then $\chi_i$ converges weakly to $\chi$, since the mapping sending a \cadlag\ function $f$ to its first jump of size greater than $a$ is continuous on functions having no jumps of size exactly $a$. 
However, note that $\chi_i$ has the law of a geometric with success parameter 
$\overline \mu_n(a)=\imf{\mu_n}{(a,\infty]}$. 
Hence, $\chi_i/a_{n_i}$ converges in law to an exponential random variable with parameter $\mu(a,\infty]$. 
On the other hand, the law of $\chi$ is exponential with parameter $\imf{\tilde\mu}{(a,\infty]}$, which implies $c=1$. 
Hence, every subsequential limit of $\tau^{n}(a_n\cdot)$ has the same law, so that $\tau^n$ converges to $\tau$ weakly. 
\end{proof}

\begin{proof}[Proof of Theorem \ref{mainLimitTheorem}]
First, note that it is sufficient to prove that $L^n/a_n$ converges to $L$ uniformly on compact sets in probability. 
(Recall that Skorohod convergence is equivalent to uniform convergence on compact sets for continuous functions. )
Then, since inverse local time $\tau$ is strictly increasing
 (because of the regularity assumption), 
 and the inverse is continuous operation on Skorohod space at strictly increasing functions (Corollary 13.6.4 in \cite{MR1876437}), 
 it suffices to prove that $\tau^n_{a_n\cdot}\to \tau$ 
 as random elements on Skorohod space in probability. 
 
By Lemmas  \ref{largeJumpConvergenceLemma} and \ref{weakConvergenceOfInverseLocalTimeLemma},  
$\tau^{n,\leq a}=\tau^n-\tau^{n,>a}$ converges weakly on Skorohod space to 
$\tau^{\leq a}=\tau-\tau^{>a}$. 
Let $d$ be the drift of $\tau$. 
Then $\tau^{n,\leq a}-d\id \to \tau^{\leq a}-d\id$ weakly on Skorohod space. 
Since, $\tau-d\id$ is a driftless subordinator, it equals the sum of its jumps. 
Hence,  $\proba{\imf{\tau^{\leq a}}{t}-dt>\eps}\to 0$ as $a\to 0$. 
Finally, if $\rho_t$ denotes the Skorohod metric on $[0,t]$, then (by the triangular inequality, Lemma \ref{largeJumpConvergenceLemma} and the definition of $\rho_t$):
\begin{linenomath}
\begin{align*}
&\limsup_{n\to\infty}\proba{\imf{\rho_t}{\tau, \tau^{n}_{a_n\cdot}}>\eps}
\\&\leq \limsup_n\bra{\proba{\imf{\rho_t}{\tau^{>a}, \tau^{n,>a}_{a_n\cdot}}>\eps/2}
+\proba{\imf{\rho_t}{\tau^{\leq a}, \imf{\tau^{n,\leq a}}{a_n\cdot}}>\eps/2}}
\\ &\leq \limsup_n \proba{\sup_{s\leq t}\abs{\tau^{n,\leq a}_{a_ns}-ds}>\eps/4}
+\proba{\tau^{\leq a}_{t}-dt>\eps/4}
\\ &\leq 2\proba{\imf{\tau^{\leq a}}{t}-dt>\eps/4}. 
\end{align*}\end{linenomath}As we just remarked, 
the right-most expression can be made as small as required by choosing $a$ small enough. 
\end{proof}

\section{Convergence of local times at zero of random walks and Galton-Watson processes with immigration}
\label{SectionOnApplicationsViaDistributionalMethods}
In this section, we present two further applications of Theorem \ref{mainLimitTheorem}, 
stated as Theorems \ref{randomWalkLocalTimeLimitTheorem} and \ref{TheoremOnConvergenceOfLocalTimesOfGWIProcesses}. 
The difference is that now the simple pathwise methods presented in Section \ref{SectionOnPathwiseApplicationsOfMainTheorem} are not applicable 
and one needs a more elaborate argument to see that the conditions of Theorem \ref{mainLimitTheorem} hold. 
The strategy is similar in both cases, 
mainly relying on the weak convergence of the hitting times of the regenerative point zero, 
which allows us to use the Markov property 
in order to deduce convergence in probability of right endpoints of excursions in an adequate probability space. 
To deduce convergence in probability of left-endpoints, 
we rely on the use of Markovian bridges and their reversibility properties. 

\subsection{Convergence of local times at zero of random walks in the domain of attraction of a stable process}
\label{SubsectinOnConvergenceOfLocalTimesAtZeroOfRandomWalks}

For the proof of Theorem \ref{randomWalkLocalTimeLimitTheorem} 
the strategy 
is to first analyze the weak convergence of hitting times of random walks 
and then use this information 
to prove that left and right hand endpoints of excursions straddling a fixed time 
converge in probability on an adequate probability space. 
The analysis of the hitting times is based on local limit theorems for random walks 
(as in \cite{MR0062975}) 
and on relationships between hitting times and resolvents.

\begin{lemma}
\label{randomWalkHittingTimeLimitLemma}
Under the assumptions of Theorem \ref{randomWalkLocalTimeLimitTheorem} and for any $x\neq 0$,  
let $H_0(\floor{b_nx}+X^1)$ be the hitting time of $0$ of $\floor{b_nx}+X^1$. 
Then $H_0(\floor{b_nx}+X^1)/n$ converges weakly to the hitting time $H_0$ of $0$ of $x+X$. 
This result also holds if $\floor{b_nx}+X^1$ and $x+X$ are conditioned to pass through zero 
at times $\floor{tn}$ and $t$ respectively. 
\end{lemma}
The conditioning of  $X$  can be formalized through the notion of a Markovian bridge, 
which is a weakly continuous disintegration of the law of $x+X$ given $X_t=y$. 
See Theorem 1 in \cite{MR2789508}; 
the construction of bridges will be used in the proof of Lemma \ref{randomWalkHittingTimeLimitLemma}. 
We will need some preliminaries for the proof, 
including the fact $0$ is regular  for our stable L\'evy process $X$ 
($0$ is also instantaneous since $X$ is not a compound Poisson process). 
We recall one possible argument since it will be relevant to our analysis. 
Start with the formula $\abs{\esp{e^{iu X_t}}}=e^{-cu^\alpha}$ valid for some $c>0$. 
Using Fourier inversion, we then see that $X_t$ admits a bounded and infinitely differentiable density $f_t$; 
using the self-similarity of stable procesess and writing $f=f_1$, 
we find that $f_t(x)=\imf{f}{xt^{-1/\alpha}}t^{-1/\alpha}$.   
Since $X$ is of unbounded variation, \cite{MR0240850} implies that $f$ is everywhere positive. 
We also infer that the semigroup of $X$ admits the transition densities $p_t(x,y)$ 
given by $\imf{p_t}{x,y}=\imf{f}{(y-x)t^{-1/\alpha}}t^{-1/\alpha}$; 
using the fact that $f$ is bounded, say by $M>0$, and that $\alpha>1$ we deduce the positivity,  
finitude and bicontinuity of the resolvent density\begin{lesn}
\imf{u_\lambda}{x,y}=\int_0^\infty e^{-\lambda t}\imf{p_t}{y-x}\, dt. 
\end{lesn}We will write $\imf{u_\lambda}{x}$ for $\imf{u_\lambda}{x,0}$. 
Let $T_\lambda$ be an exponential random variable of parameter $\lambda$ independent of $X$. 
Then the resolvent operator $U_\lambda$ of $X$ 
is defined by $\lambda U_\lambda f(x)=\esp{\imf{f}{x+X_{T_\lambda}}}$ 
and we can write\begin{lesn}
\imf{U_\lambda f}{x}= \int_0^\infty \imf{u_\lambda}{x,y}\imf{f}{y}\, dy. 
\end{lesn}The notation $\imf{U_\lambda}{x,A}$ will stand for $\imf{U_\lambda\indi{A}}{x}$. 

Since the L\'evy measure of $X$ has no atoms, we see that $X$ does not jump into zero. 
Let $B_\eps=(-\eps,\eps)$ and $H_\eps$ be the hitting time of $B_\eps$. 
Using the Markov property, we see that\begin{lesn} 
\imf{U_\lambda}{x, B_\eps}=\imf{\se_x}{e^{-\lambda H_\eps}\imf{U_\lambda}{X_{H_\eps},B_\eps }}.
\end{lesn}As $\eps\to 0$, $H_\eps$ increases, say to $H_0$. 
Since $X$ does not jump into zero, we must have $H_\eps<H_0$ for small enough $\eps$. 
Hence, $H_0$ is predictable and hence, by quasi-continuity of $X$, 
$X$ is continuous at $H_0$ so that $H_0$ equals the hitting time of $0$ of $X$. 
Also, taking limits as $\eps\to 0$ in the preceding display and using the positivity and bicontinuity of the resolvent density, one obtains
\begin{lesn}
\imf{\se_x}{e^{-\lambda H_0}}
=\frac{\imf{u_\lambda}{x}}{\imf{u_\lambda}{0}}
>0. 
\end{lesn}Hence for any $x\neq 0$, $H_0$ is finite with positive probability 
under any $\p_x$, so that $0$ is not polar. 
By self-similarity, $\imf{\p_x}{H_0<\infty}$ is independent of $x$; 
denote it by $\rho$. 
Let $R$ be the first return time to $0$. 
By the Markov property, we see that $\imf{\p_0}{R<\infty}\geq \rho$. 
Hence, there exists $t>0$ such that $\imf{\p_0}{R\leq t}\geq \rho/2$. 
By self-similarity, $\imf{\p_0}{R\leq t}$ does not depend on $t$, 
so that $\imf{\p_0}{R=0}\geq \rho/2$ and by the Blumenthal 0-1 law, 
we see that $\imf{\p_0}{R=0}=1$. 
Hence $0$ is a regular and instantaneous point. 
(The instantaneous character follows from the fact that $X$ is not compound Poisson.) 

We will make use of the following local limit theorem. 
\begin{theorem}[\cite{MR0048731}, \cite{MR0062975}]
Let $f$ be the density of $X_1$. 
Under the assumption that the law of $X^1_1$ has span $1$:
\begin{lesn}
\lim_{n\to\infty}\sup_{k\in \z}\abs{b_n\proba{X^1_n=k}-\imf{f}{\frac{k}{b_n}}}=0. 
\end{lesn}
\end{theorem}

\begin{proof}[Proof of Lemma \ref{randomWalkHittingTimeLimitLemma}]
Let $e_n=\sup_{k\in\z}|b_n\proba{X^1_n=k}-\imf{f}{k/b_n}|$. 
From Theorem 2 in \cite{MR0138128}, we see that $b_n$ is regularly varying of index $1/\alpha$. 
We then use Potter's bounds (for sequences) as in \cite[Theorem 1.5.6]{MR898871}: 
for any $\delta >0$ there exists $A>1$ such that\begin{lesn}
b_n/b_k\leq A \bra{(n/k)^{\delta+1/\alpha}\vee (n/k)^{-\delta+1/\alpha}}. 
\end{lesn}Note that if $k/n\to t>0$ as $n\to\infty$, then $b_k/b_n\sim (k/n)^{1/\alpha}\to t^{1/\alpha}$. 

To get a limit theorem for the hitting times of $0$, 
consider the scaled random walk
 $X^n_t=X^1_{\floor{nt}}/b_n$, which jumps every $1/n$ and has span $1/b_n\to 0$. 
 Note that $H_0(\floor{b_nx}+X^1)/n$ equals the hitting time of $0$ of $\floor{b_nx}/b_n+X^n$. 
We then use the 
random variables $T=T_\lambda$ and $T_n=\ceil{nT}/n$ 
to define a (type of) resolvent\begin{lesn}
\lambda\imf{u^n_\lambda}{x,y}=\imf{\se_x}{\int_0^\infty 
\lambda e^{-\lambda t}\indi{X^1_{\ceil{nt}}/b_n=y}}= \imf{\p_x}{X^n_{T_n}=y}
\end{lesn}for $x,y\in \z/b_n$. (The inequality $T_n\geq T$ simplifies some arguments, which is why it was chosen and explains why $u^n_\lambda$ is referred to as \emph{a} resolvent. )
We again write $\imf{u^n_\lambda}{x}=\imf{u^n_\lambda}{x,0}$. 
As before, note that if $H^n_0$ is the hitting time of $0$ of $X^n$ and $x\neq 0$, 
we have
\[
	\imf{\se_x}{e^{-\lambda H^n_0}}=\frac{\lambda u^n_\lambda(x)}{u^n_\lambda(0)+e^{\lambda/n}-1}
\]
We now wish to see that $H^n_0\to H_0$ in distribution, 
where $H^n_0$ is taken under $\p_{x_n}$ for any $x_n\in \z/b_n$ 
such that $x_n\to x\neq 0$. 
This will follow if we prove that $b_n\imf{u^n_\lambda}{x_n}\to \imf{u_\lambda}{x}$. 
Indeed, it implies first that $u^n_{\lambda}(0)/n\sim 0$ since $b_n$ is regularly varying of index $1/\alpha$ and $\alpha\in (1,2)$, and then the convergence of the Laplace transform of $H^n_0$. 

Let us now prove that $b_n\imf{u^n_\lambda}{x_n}\to \imf{u_\lambda}{x}$ 
for any $x\in\re$. 
Note first that\begin{lesn}
\lambda\imf{u_\lambda}{x}= \esp{\imf{f_T}{x}}=\esp{\imf{f_1}{xT^{-1/\alpha}}T^{-1/\alpha}}. 
\end{lesn}%
Since $T_n\to T$ and $T_n\geq T$, we see that $T_n^{-1/\alpha}\leq T^{-1/\alpha}$ 
and the latter variable has finite expectation since $1/\alpha\in [1/2,1)$. 
Since $T_n$ converges almost surely to the positive random variable $T$, 
then $\imf{f}{b_n x_n/b_{nT_n}}b_n/b_{nT_n}$ converges almost surely 
to $\imf{f}{x T^{-1/\alpha}}T^{-1/\alpha}$. 
Using Potter's bounds and the boundedness of $f$, 
we see that for any $\delta \in (0,1/\alpha)\subset(0,1)$
\begin{linenomath}
\begin{align*}
\imf{f}{b_n x_n/b_{nT_n}}b_n/b_{nT_n}
&\leq A' \bra{T_n^{\delta-1/\alpha}\vee T_n^{-\delta-1/\alpha}}
\\&\leq A'\bra{ T^{\delta-1/\alpha}\vee T^{-\delta-1/\alpha}}. 
\end{align*}\end{linenomath}The latter random variable has finite expectation. 
Hence, we can use dominated convergence to conclude that
\begin{linenomath}
\begin{align*}
\sum_{k} e^{-\lambda k/n}\paren{1-e^{\lambda/n}} \imf{f}{x_nb_n/b_{k+1}}b_n/b_{k+1}
&=\esp{\imf{f}{b_n x_n/b_{nT_n}}b_n/b_{nT_n}}
\\&\to \esp{\imf{f_1}{xT^{-1/\alpha}}T^{-1/\alpha}}
=\imf{u_\lambda}{x}. 
\end{align*}\end{linenomath}On the other hand, note that
\begin{linenomath}
\begin{align*}
&\abs{b_n \imf{u^n_\lambda}{x_n}- \esp{\imf{f}{b_nx_n/b_{nT_n}}b_n/b_{nT_n}} }
\\&\leq \sum_k  e^{-\lambda k/n}\paren{1-e^{\lambda/n}} 
\abs{b_n\proba{X^1_{k+1}=b_nx_n} -  \imf{f}{b_nx_n/b_{k+1}}b_n/b_{k+1}}
\\&\leq \sum_k e^{-\lambda k/n}\paren{1-e^{\lambda/n}} b_ne_{k+1}/b_{k+1}
\\&\leq \sum_k e^{-\lambda k/n}\paren{1-e^{\lambda/n}} e_{k+1}A\bra{\frac{1}{(k+1/n)^{1/\alpha-\delta}}\vee\frac{1}{(k+1/n)^{1/\alpha+\delta}}}
\\&=\esp{T_n^{-1/\alpha-\delta}\vee T_n^{-1/\alpha+\delta}e_{nT_n}}.
\end{align*}\end{linenomath}Since $T_n^{-1/\alpha-\delta}\vee T_n^{-1/\alpha+\delta}e_{nT_n}\to 0$ as $n\to\infty$ and
\[
	T_n^{-1/\alpha-\delta}\vee T_n^{-1/\alpha+\delta}e_{nT_n}\leq CT^{-1/\alpha-\delta}\vee T^{-1/\alpha+\delta}, 
\]we can apply dominated convergence to conclude that\begin{lesn}
\abs{b_n \imf{u^n_\lambda}{x_n}- \esp{\imf{f}{b_nx_n/b_{nT_n}}b_n/b_{nT_n}} }\to 0.
\end{lesn}We therefore conclude that $b_n\imf{u^n_\lambda}{x_n}\to \imf{u_\lambda}{x}$ and that therefore $H^n_0\to H_0$ as $n\to\infty$. 

We now prove the weak convergence of hitting times for bridges. 
First, in the discrete setting, 
consider the law $\p_{x_n,y_n}^{\floor{nt},n}$ of $(X^n_{s},s\leq t)$ under $\p^n_{x_n}$ conditioned on $X^n_t=y_n$. 
(The conditioning event has positive probability for large enough $n$ 
by the local limit theorem and positivity of stable densities. ) 
Under the law $\p_{x_n,0}^{\floor{nt}}$, the hitting time $H^n_0$ is always finite and bounded by $t$. 
The following absolute continuity relationship follows immediately from the Markov property: 
for every $s\leq t$ and any $A\in \F^n_s=\sag{X^n_r:r\leq s}$ we have
\begin{lesn}
\imf{\p_{x_n,y_n}^{\floor{nt},n}}{A}
=\imf{\se^n_{x_n}}{\indi{A}\frac{\imf{p^n_{t-s}}{X^n_s, y_n}}{\imf{p^n_t}{x_n,y_n}}}.
\end{lesn}Using the reasoning of Proposition VIII.1.3 of \cite{MR1725357}, 
the above absolute continuity relationship can be extended to the hitting time $H^n_0$, 
so that for every $A\in \F^n_{H^n_0}$
\begin{lesn}
\imf{\p_{x_n,0}^{\floor{nt},n}}{A}
=\imf{\p_{x_n,0}^{\floor{nt},n}}{A\indi{H^n_0\leq t}}
=\imf{\se^n_{x_n}}{\indi{A}\frac{\imf{p^n_{t-H^n_0}}{0,0}}{\imf{p^n_t}{x_n,0}}\indi{H^n_0\leq n}}. 
\end{lesn}In particular, we see that\begin{lesn}
1=\imf{\se^n_{x_n}}{\frac{\imf{p^n_{t-H^n_0}}{0,0}}{\imf{p^n_t}{x_n,0}}\indi{H^n_0\leq n}}. 
\end{lesn}

The same results hold in continuous space, made rigorous using the theory of \cite{MR2789508}. 
Namely, the bridge measures $\p_{x,y}^t$ are determined by the fact that for any $s<t$ and every $A\in \F_s$ 
we have the local absolute continuity relationship\begin{lesn}
\imf{\p_{x,y}^t}{A}
=\imf{\se_{x}^t}{\indi{A}\frac{\imf{p_{t-s}}{X_s,y}}{\imf{p_t}{x,y}}},
\end{lesn}as well as their weak continuity with respect to $x$ and $y$. 
Furthermore, the image of $\p_{x,y}^t$ under the time-reversal operation 
$X\mapsto (X_{(t-s)-},0\leq s\leq t)$ 
is the bridge from $y$ to $x$ in $t$ units of time of the dual process $-X$ 
(weak continuity is helpful at this point in order to inherit the time-reversal property 
from the L\'evy processes to their bridges). 
Since under the law $\hat \p_0$, $0$ is regular, we see that the last zero of $\hat X$ before time $t$ 
is strictly positive. 
Therefore, under the measure $\p_{x,0}^t$, $H_0<t$ almost surely. 
We can therefore apply Proposition VIII.1.3 of \cite{MR1725357} to obtain that, 
for every $A\in \F_{H_0}$
\begin{lesn}
\imf{\p_{x,0}^t}{A}
=\imf{\p_{x,0}^t}{A,H_0<t}
=\imf{\se_x}{\indi{A}\indi{H_0<t}\frac{\imf{p_{t-H_0}}{0,0}}{\imf{p_t}{x,0}}}. 
\end{lesn}In particular, note that\begin{lesn}
1=\imf{\se_x}{\indi{H_0<t}\frac{\imf{p_{t-H_0}}{0,0}}{\imf{p_t}{x,0}}}. 
\end{lesn}

Assume now that the weak convergence of $H^n_0\to H$ for the random walk hitting times holds almost surely. 
By Scheff\'e's lemma and the continuity in the time variable  of the transition densities, 
which holds by scaling, 
we deduce that\begin{linenomath}
\begin{align*}
\imf{\se_{x_n,0}^{t,n}}{e^{-\lambda H^n_0}}
&=\imf{\se_{x_n}}{e^{-\lambda H^n_0}\indi{H^n_0\leq t}\frac{\imf{p^n_{t-H^n_0}}{0,0}}{\imf{p^n_t}{x_n,0}}}
\\&\to \imf{\se_{x}}{e^{-\lambda H_0}\indi{H_0<t}\frac{\imf{p_{t-H_0}}{0,0}}{\imf{p_t}{x_n,0}}}
=\imf{\se_{x,0}^{t}}{e^{-\lambda H_0}}. 
\end{align*}\end{linenomath}We deduce the stated weak convergence. 
\end{proof}

\begin{proof}[Proof of Theorem \ref{randomWalkLocalTimeLimitTheorem}]
Let $X^n_{t}=X^1_{\floor{nt}/n}/b_n$. 
Thanks to a theorem of Skorohod, our assumption tells us that $X^n$ converges weakly to $X$. 
Again by a theorem of Skorohod, we will assume that the convergence holds almost surely in some probability space. 
In particular, since $X$ is continuous almost surely at any deterministic $t\geq 0$, we see that $X^n_t\to X_t$  almost surely. 
Also, since the law of $X_t$ is absolutely continuous, we see that $X_t\neq 0$ almost surely so that $t<d_t$ almost surely. 

Let $d^n_{\eps,t}$ and $d_{\eps,t}$ denote the first times that $X^n$ and $X$ enter the sets $(-\eps,\eps)$ after time $t$. 
Let us also define $g^n_{\eps,t}$ and $g_{\eps,t}$ as the last times the left limits of 
$X^n$ and $X$ belong to the set $(-\eps,\eps)$. 
(The asymmetry in the definitions is required so that,  when using time-reversal, the arguments become symmetric.) 
Consider the inequality\begin{lesn}
\proba{\abs{d_t-d^n_t}>\delta}
\leq 
\proba{\abs{d_t-d_{\eps,t}}>\delta/3}
+\proba{\abs{d_{\eps,t}-d^n_{\eps,t}}>\delta/3}
+\proba{\abs{d^n_{\eps,t}-d^n_t}>\delta/3}. 
\end{lesn}and the corresponding one for the $g$-type random variables. 
We now prove that three summands in the right-hand side of the above inequality tend to zero if one first takes limit as $n\to\infty$ and then as $\eps\to 0$. 
This will prove the convergence in probability needed to apply Theorem \ref{mainLimitTheorem} at the end of the proof. 
\begin{description}
\item[First summand] 
	This argument has been given before, 
	when proving that stable processes are continuous at their hitting time of zero, 
	just before the proof of Lemma \ref{randomWalkHittingTimeLimitLemma}. 
	Note that $d_{\eps,t}\leq d_{\eps',t}\leq d_t$ if $\eps'\leq \eps$ 
	and that $d_{\eps,t}$ is a stopping time for any $t$. 
	Since the L\'evy measure of $X$ has no atoms, we see that $X$ does not jump into zero. 
	It follows that $d_{\eps,t}$ has a limit $\tilde d_t$ which cannot equal any $d_{\eps,t}$. 
	By quasi-continuity of $X$, it follows that 
	$0=\lim_{\eps\downarrow 0}X_{d_{\eps,t}}=X_{\tilde d_t-}= X_{\tilde d_t}$, 
	so that $\tilde d_t=d_t$ and hence $d_{\eps,t}\to d_t$ almost surely as $\eps\downarrow 0$. 
	In particular:
	\[
		\lim_{\eps\to 0}\lim_{n\to\infty}\proba{\abs{d_t-d_{\eps,t}}>\delta/3}=0\text{ for all }\delta>0.
	\]
	The argument is slightly simpler for $g_t$ and $g_{\eps,t}$ by right continuity. 
	Indeed, we again have the inequalities $g_t\leq g_{\eps',t}\leq g_{\eps,t}$ if $0<\eps'<\eps$. 
	Hence $g_{\eps,t}$ decreases to a limit $\tilde g_t$ as $\eps\downarrow 0$. 
	By right-continuity, we see that $X_{\tilde g_t}=0$, 
	so that $g_t=\tilde g_t$ and therefore $g_{\eps,t}\to g_t$ in probability. 
	In particular: 
	\[
		\lim_{\eps\to 0}\lim_{n\to\infty}\proba{\abs{g_t-g_{\eps,t}}>\delta/3}=0\text{ for all }\delta>0.
	\]
\item[Second summand] 
	We are interested in the convergence of the hitting time of an open set after time $t$. 
	By the Markov property and Lemma \ref{SkorohodSpaceHittingTimeOfOpenSetLemma}, 
	we only need to prove that 
	if
	\[
		\tilde T_{\eps}=\inf\set{t\geq 0: X_{t}\in\set{-\eps,\eps}}, 
		\tilde T^-_\eps=\inf\set{t\geq 0: X_{t-}\in\set{-\eps,\eps}}
	\]and $T_\eps=\inf\set{t\geq 0: X_t\in (-\eps,\eps)}$, 
	then $T_\eps,\leq \tilde T_\eps, \tilde T^-_\eps$ almost surely under any $\p_x$. 
	Since $\tilde T^-_\eps$ is a predictable stopping time then $X$ is continuous at this time. 
	As in the previous item, we also see that $\tilde T^-_\eps=\tilde T_\eps$. 
	However, by the strong Markov property and regularity of both half-lines for stable processes, 
	we see that on each  the sets $\tilde T_\eps\leq T_\eps$ and $\tilde T^-_\eps\leq T_\eps$ , 
	we actually have $\tilde T^-_\eps=\tilde T_\eps=T_\eps$. 
	We therefore conclude that the hitting time of $(-\eps, \eps)$ after $t$ is continuous 
	(as a functional on Skorohod space by Lemma \ref{SkorohodSpaceHittingTimeOfOpenSetLemma})
	at almost every trajectory of $X$. 
	We conclude that $d^n_{\eps,t}\to d_{\eps,t}$ almost surely. 
	In particular
	\[
		\lim_{\eps\to 0}\lim_{n\to\infty} \proba{\abs{d^n_{\eps,t}-d_{\eps,t}}>\delta/3}
		=0. 
	\]
	
	In the case of $g_{\eps, t}$, consider the time-reversed process (from $t$) defined by $\hat X_s=X_{(t-s)-}$. 
	Then $g_{\eps, t}$ becomes the hitting time of $(-\eps,\eps)$ of $\hat X$ 
	and we will deduce that $g_{\eps, t}$ is continuous at almost all sample paths 
	by using Lemma \ref{SkorohodSpaceHittingTimeOfOpenSetLemma}. 
	For this, we note that the law of $\hat X$ given $X_t=x$ is $\hat \p_{x,0}^t$. 
	Now, we can use almost the same reasoning as before, since if $\hat X$ reaches $\{-\eps,\eps\}$ at $s<t$, 
	then $(\hat X_{s+r},r\leq s-t)$ has law $\hat\p_{\hat X_s,0}^{t-s}$, and under this law, 
	$(-\eps, \eps)$ is reached immediately by regularity. 
	We conclude that $g^n_{\eps,t}\to g^n_{\eps,t}$ almost surely. 
	In particular, 
	\[
		\lim_{\eps\to 0}\lim_{n\to\infty} \proba{\abs{g^n_{\eps,t}-g_{\eps,t}}>\delta/3}
		=0. 
	\]
\item[Third summand] Using the strong Markov property for $X^n$, 
	we see that $d^n_{t}-d^n_{\eps,t}$ 
	has the same law as the hitting time of zero starting from $X^n_{d^n_{\eps,t}}$, 
	so that in particular
	\begin{lesn}
		\esp{e^{-\lambda [d^n_{t}-d^n_{\eps,t}]}}
		=\esp{\imf{\se_{X^n_{d^n_{\eps,t}}}}{e^{-\lambda H^n_0}}}. 
	\end{lesn}However as we proved for the second summand that $T_\eps\leq \tilde T_\eps$, 
	Lemma \ref{SkorohodSpaceHittingTimeOfOpenSetLemma} implies that 
	$X^n_{d^n_{\eps,t}}\to X_{d_{\eps,t}}$ as $n\to\infty$ so that, by Lemma \ref{randomWalkHittingTimeLimitLemma},  
	\[
		\esp{\imf{\se_{X^n_{d^n_{\eps,t}}}}{e^{-\lambda H^n_0}}}
		\to \esp{\imf{\se_{X_{d_{\eps,t}}}}{e^{-\lambda H_0}}}
		=\esp{e^{-\lambda (d_{t}-d_{\eps,t})}}. 
	\]
	We conclude that
	\[
		\lim_{\eps\to 0}\lim_{n\to\infty}\proba{d^n_{t}-d^n_{\eps,t}>\delta/3}
		=\lim_{\eps\to 0}\proba{d_t-d_{\eps,t}>\delta/3}
		=0.
	\]
	Regarding $g^n_{\eps,t}-g^n_t$, 
	we use the backward strong Markov property (cf. Theorem 2 in \cite{MR2789508}) 
	which states that, given $g^n_{\eps, t}=s$ and $X^n_{s}=y$, 
	the law of $\paren{X^n_r,r\leq s}$ is the bridge law $\p_{0,y}^{s,n}$. 
	(It is simple to prove this for discrete time Markov chains.)
	Using time-reversal, we can then write\begin{lesn}
	\esp{e^{-\lambda \bra{g^n_{\eps,t}-g^n_t}}}
	=\esp{\imf{\hat\se_{X^n_{g^n_{\eps,t}},0 }^{t-g^n_{\eps,t},n} }{e^{-\lambda H_0}}}. 
	\end{lesn}Applying Lemma \ref{SkorohodSpaceHittingTimeOfOpenSetLemma} 
	for the time-reversed random walk at $t$, 
	we see that $g^n_{\eps,t}\to g_{\eps,t}$ and  $X^n_{g^n_{\eps,t}}\to X_{g_{\eps,t}}$ as $n\to\infty$. 
	But then Lemma \ref{randomWalkHittingTimeLimitLemma} tells us that 
	\[
	\esp{e^{-\lambda \bra{g^n_{\eps,t}-g^n_t}}}
	\to \esp{\imf{\hat\se_{X_{g_{\eps,t}},0 }^{t-g_{\eps,t}} }{e^{-\lambda H_0}}}
	=\esp{e^{-\lambda \bra{g_{\eps,t}-g_t}}}.
	\]By the analysis for the first summand, we conclude that 
	\[
		\lim_{\eps\to 0}\lim_{n\to\infty}\proba{g^n_{\eps,t}-g^n_{t}>\delta/3}
		=\lim_{\eps\to 0}\proba{g_{\eps,t}-g_t>\delta/3}
		=0. 
	\]
%
\end{description}

Putting the above together, we see that
\[
	\limsup_{n\to\infty} \proba{\abs{d^n_t-d_t}>\delta}=0
	\quad\text{and}\quad
	\limsup_{n\to\infty} \proba{\abs{g^n_t-g_t}>\delta}=0. 
\]

Theorem \ref{mainLimitTheorem} therefore applies and shows us that $(X^n,L^n)$ converges in probability to $(X,L)$ in this particular probability space. 
We therefore get weak convergence on any sequence of probability spaces where the random walk $S$ is defined. 
\end{proof}


\subsection{Convergence of local times of Galton-Watson processes with immigration}
\label{subsectionOnGWI}

In this subsection, we will prove Theorem \ref{TheoremOnConvergenceOfLocalTimesOfGWIProcesses}. 
Recall that $Z^1$ stands for a $\gwi(\mu,\nu)$ process started at zero, 
where $\mu$ is the geometric distribution on $\na$ of parameter $1/2$ 
and $\nu$ is the geometric distribution on $\na$ of parameter $p$ and mean $\delta=p/(1-p)\in (0,1)$. 

Let us discuss the scaling limit of $Z^1$, which will be described in terms of a Brownian motion $B$. 
Note that $\mu$ has mean $1$ and finite variance $2$. 
Therefore, if $X^1$ is a (downwards skip-free) random walk which has jumps of size $k$ 
with probability $\mu_{k+1}$, 
and $X^n_{t}=X^1_{n^2 t}/n$, it follows that $(X^n_t)$ converges weakly to  $(B_{2 t})$. 
Also, consider a random walk $Y^1$ with jump distribution $\nu$. 
Setting $Y^n_{t}=Y^1_{nt}/n$, 
a suitable extension of the strong law of large numbers tells us that $Y^n\to \delta\id$ 
(uniformly on compact sets almost surely). 
It is not hard to show that we can recursively construct $Z^1$ in terms of $X^1$ and $Y^1$ by setting
\[
	C_{-1}=0, \quad
	Z^1_m=k+X^1_{C_{m-1}}+Y^1_m \quad\text{and}\quad
	C_{m}=C_{m-1}+Z_m.
\]The above can be seen as a discrete time-change equation which possesses a strong stability theory. 
For example, Corollary 7 in \cite{MR3098685}  tells us that $Z^n$ converges 
to the unique solution to
\begin{equation}
	\label{TCEOfZeroDimSQBesselType}
	Z_t=B_{2\int_0^t Z_s\, ds}+\delta t. 
\end{equation}(By Knight's theorem, 
$Z$ is a weak solution to the (perhaps more familiar) SDE \eqref{SDEofBesselType} of the squared Bessel type 
in the statement of Theorem \ref{TheoremOnConvergenceOfLocalTimesOfGWIProcesses}.) 
And, if we assume that $X^n$ and $Y^n$ are independent 
and converge almost surely to $B$ and $\delta \id$ 
in an adequate probability space (which we assume from now on), 
the convergence of $Z^n$ to $Z$ actually holds almost surely. 
The process $Z$ is called a Continuous-State Branching Process with Immigration, 
with reproduction mechanism $\Psi$ and immigration mechanism $\Phi$, 
where $\Psi(\lambda)=\lambda^2$ and $\Phi(\lambda)=\delta \lambda$ 
are the Laplace exponents of the spectrally positive L\'evy process $B_{2\cdot}$ 
and the (deterministic!) subordinator $\delta\id$. 
This will be abridged $\cbi(\Psi,\Phi)$. 
These processes were introduced in \cite{MR0290475} as the possible large population scaling limits
 of $\gwi$ processes. 

Of course, we will obtain Theorem \ref{TheoremOnConvergenceOfLocalTimesOfGWIProcesses} 
as an application of Theorem \ref{mainLimitTheorem}. 
Recall that $0$ is regular and instantaneous for $Z$, 
so that it remains to prove that, for any $t>0$, $(g^n_t,d^n_t)\to (g_t,d_t)$ in probability as $n\to\infty$. 
The strategy of proof is similar to the random walk case, 
except that there are some simplifications that arise from assuming the specific forms 
of the offspring and immigration distribution, 
since $\mu$ is a special case of a fractional linear offspring law. 
On the other hand, 
Theorem \ref{TheoremOnConvergenceOfLocalTimesOfGWIProcesses} 
is not given a more general statement because, 
as in the random walk case, 
we make use of local limit theorems, 
existence of bridges and time-reversal, 
which have not been established for more general 
$\cbi$ processes. 

Our strategy to prove convergence in probability of endpoints of excursions follows the following steps. 
\begin{enumerate}
\item Convergence of extinction times of $\gw$ processes to extinction times of $\cb$ processes. 
\item Weak convergence of endpoints of excursions of $\gwi$ processes. 
\item Convergence of hitting times of $\gwi$ processes from suitable initial points to hitting times of zero for $\cbi$ processes (using the above two items). 
\item Convergence in probability of endpoints of excursions; 
left endpoints using a (non-uniform) local limit theorem, bridges and reversibility. 
\end{enumerate}

\begin{proposition}
\label{propositionOnConvergenceOfExtinctionTimesForOneGWProcess}
Let $\tilde U^n$ be a $\gw(\mu)$ process started at $k_n$, 
$U^n_{t}=\tilde U^{n}_{\floor{nt}}/n$ 
and assume that $k_n/n\to z\geq 0$. 
Then, the extinction time of $U^{n}$ converges weakly as $n\to\infty $ 
to the extinction time of the solution $U$ of \eqref{TCEOfZeroDimSQBesselType} with $\delta=0$. 
\end{proposition}

\begin{proof}
Recall that the extinction time of a $\gw$ or a $\cb$ is the hitting time of $0$, 
when the process becomes absorbed. 

Following the analysis of \cite[\S 3]{MR0408016}, the extinction time $T_0$ of $U$ satisfies
\[
	\proba{T_0\leq t}=\proba{U_t=0}=e^{-z/t}. 
\]On the other hand, let $f$ be the generating function of $\mu$, given by
\[
	\imf{f}{s}=\frac{1}{2-s}. 
\]Then, the iterates of $f$ admit the following simple form
\[
	f^{\circ n}(s)=\underbrace{f\circ\cdots\circ f}_{\text{$n$ times}}(s)=\frac{n - (n - 1)s}{n + 1 - ns}
\]As is well known, $\p(\tilde U^{n}_m=0)=f^{\circ m}(0)^{k_n}$. 
Therefore
\[
	\proba{U^{n}_t=0}
	= f^{\circ \floor{nt}}(0)^{nz}
	=\bra{1-\frac{1}{\floor{nt}+1}}^{nz}
	\to e^{-z/t}.\qedhere
\]
\end{proof}

\begin{proposition}
\label{PropositionOnWeakConvergenceOfExcursionEndpointsForGWI}
Under the setting of Theorem \ref{TheoremOnConvergenceOfLocalTimesOfGWIProcesses}, 
the random variables $(g^n_t,d^n_t)$ converge weakly as $n\to\infty$ to $(g_t,d_t)$. 
\end{proposition}
The proof is based on an analysis of the zero set of a GWI or a CBI process. 
The structure of the zero set of the CBI process $Z$, which is that of a random cutout set. 
Random cutouts were introduced in \cite{Mandelbrot1972} and further studied in \cite{MR799145} 
(see also the account on \cite[Ch. 7]{MR1746300}). 
They were connected to CBI processes 
(and in particular squared Bessel type ones of equation \ref{SDEofBesselType}) 
in \cite{MR3263091}. 
The random cutout structure is also found and easily understood in the setting of GWI processes as follows. 
Random cutouts (on $\na$) are constructed out of a sequence of iid random variables $(L_i,i\geq 0)$ 
with values in $\na$  
which are used to remove the integer intervals $\set{j: i\leq j<i+L_i}$  from $\na$ 
to get the uncovered set
\[
	\mc{U}=\na\setminus\bigcup_{i\geq 0}\set{j: i\leq j<i+L_i}. 
\]To related it to $\gwi$ processes, let $f$ and $g$  be the offspring  and immigration generating functions. 
Recall that the probability that a $\gw_1(f)$ process is extinct by time $n$ 
equals 
 $f^{\circ n}(0)$. 

Let $L_0$ be constructed as follows: 
let $K_0$ have generating function $g$ and, 
conditionally on $K_0=k$, 
let $L_0$ have the law of the extinction time of a GW process 
with offspring generating function $f$ which starts at $k$. 
Note that
\begin{equation}
\label{equation:DiscreteCutoutDistributionGWI}
\proba{L_0\leq n}=\imf{g\circ f^{\circ n}}{0}. 
\end{equation}We then let $(L_i)$ be iid with the same law as $L_0$. 
The interpretation is that $L_i$ stands for 
the number of generations spanned by  the descendants of immigrants arriving at generation $i$. 
Then, it follows easily that the random cutout set $\mc{U}$ based on the sequence $L$ 
has the same law as the zero set of a $\gwi(\mu,\nu)$ started at zero. 

\begin{proof}[Proof of Proposition \ref{PropositionOnWeakConvergenceOfExcursionEndpointsForGWI}]
An important computation for random cutouts is that of the renewal density
Note that\[
	\proba{n\in \mc{U}}
	=\prod_{m=0}^n \proba{L_1\leq n-m}
	=\prod_{m=0}^n g\circ f_{n-m}(0). 
\]In our case, we obtain the asymptotics
\[
	\proba{n\in \mc{U}}
	=e^{-\sum_{m=0}^n -\log(1-\frac{1-p}{1+pm})}
	\sim \gamma n^{-\delta}
\]Let $U(x)=\sum_{n\leq x} \proba{n\in \mc{U}}$. The above displays tells us that $U(x)\sim x^{1-\delta}$, 
so that its Laplace transform, $\hat U$, satisfies $\hat U(x)\sim x^{-(1-\delta)}$. 
If $R_1$ has the law of the first return time to zero of a $\gwi(\mu,\nu)$ started at $0$, formula \eqref{LaplaceTransformLeftEndpointOfExcursionDiscreteTime1} states that
\[
	\frac{1}{1-\esp{e^{-\lambda R_1}}}= \sum_{n}\proba{n\in \mc{U}} e^{-\lambda n},
\]so that $1-\esp{e^{-\lambda R_1}}\sim c\lambda^{1-\delta}$ as $\lambda\xrightarrow{} 0$. 
Let now $R$ be a random walk based on the law of $R_1$. 
It follows that $R$ is in the domain of attraction of a stable subordinator $\tau$ of index $1-\delta$: the process $R^n$ given by $R^n_t=R_{n^{1-\delta} t}/n$ converges weakly to $\tau$. 
If $T$ denotes a standard exponential random variable independent of $R$ 
and $\tau$, let $g^n_t\leq t<d^n_t$ (resp. $g_t\leq t<d_t$) 
be the points on the range of $R$ and $\tau$ closest to $t$ 
(note that $g^n_t=n^{-(1-\delta)}g^1_{n^{1-\delta} t}$), then, 
because of \eqref{leftEndpointSubLTEq} and  \eqref{LaplaceTransformLeftEndpointOfExcursionDiscreteTime1}, 
\begin{linenomath}
\begin{align*}
	\esp{e^{-\alpha g^n_{T/\theta}-\beta d^n_{T/\theta}}}
	&=\frac{\esp{e^{-n^{-(1-\delta)}\beta R}-e^{-n^{-(1-\delta)}(\beta+\theta)R}}}{\esp{1-e^{n^{-(1-\delta)}(\alpha+\beta+\theta)R}}}
	\\&\xrightarrow[n\to\infty]{}
	\frac{\paren{\beta+\theta}^{1-\delta}-\paren{\beta}^{1-\delta}}{\paren{\alpha+\beta+\theta}^{1-\delta}}
	=\esp{e^{-\alpha g_{T/\theta}-\beta d_{T/\theta}}}
\end{align*}\end{linenomath}

We now prove that, for every $t>0$, $d^n_t\to d_t$ and $g^n_t\to g_t$ weakly. 
Notice first that, for $s\leq t$, the excursion interval straddling $s$ ends after $t$ 
if and only if the excursion interval straddling $t$ begins before $s$. 
Therefore, from the above paragraph, we see that, 
\begin{linenomath}
\begin{align*}
	\int e^{-\theta s} \proba{d^n_t\in ds}
	&=\int \theta e^{-\theta s} \proba{s\leq d^n_t}\, ds
	=\int \theta e^{-\theta s}\proba{g^n_s\leq t}\, ds
	=\proba{g^n_{T/\theta}\leq t}
	\\&\to \proba{g_{T/\theta}\leq t}
	=\int \theta e^{-\theta s}\proba{g_s\leq t}\, ds
	=\int \theta e^{-\theta s} \proba{s\leq d_t}\, ds
	\\&=\int e^{-\theta s} \proba{d^n_t\in ds}
\end{align*}\end{linenomath}We conclude the weak convergence $d^n_t\to d_t$ and a similar argument gives us the weak convergence $g^n_t\to g_t$. 
\end{proof}

In particular, the above proposition proved weak convergence of the inverse local time of $Z^n$ 
with a precise scaling sequence rather than the abstract one of Theorem \ref{mainLimitTheorem}. 
That is what explains the more explicit scaling in Theorem \ref{TheoremOnConvergenceOfLocalTimesOfGWIProcesses}. 

\begin{proposition}
\label{propositionOnWeakConvergenceOfGWIHittingTimesofZero}
Let $\tilde Z^n$ be a $\gwi(\mu,\nu)$ that starts at $k_n$ and $Z^n_t=\tilde Z^n_{\floor{nt}}/n$. 
If $k_n/n\to z$ as $n\to\infty$ and $Z$ is the unique solution to \eqref{SDEofBesselType} started at $z$, 
then the hitting time of zero of $Z^n$ converges to that of $Z$. 
\end{proposition}

\begin{proof}
By the branching property for $\gwi$ processes, 
we write $\tilde Z^n=U+\tilde V^n$, where $U$ is a $\gwi(\mu,\nu)$ started at zero 
and $\tilde U^n$ is an independent $\gw(\mu)$ that starts at $k_n$ 
as in Proposition \ref{propositionOnConvergenceOfExtinctionTimesForOneGWProcess}. 
We also let $Z^n_t=\tilde Z^n_{\floor{nt}}/n$ and define $U^n$ and $V^n$ by applying the latter scaling 
to $U$ and $\tilde V^n$. Note that we can also write $Z^n$ 
as the sum $U^n+V^n$ of independent processes. 
We can also perform the above decomposition for $Z=U+V$. 

Let $\tilde T^n$ be the extinction time of $V^n$, $T^n$ be the hitting time of zero for $Z^n$ and $d^n_t$ 
be the first zero of $U^n$ after $t$. 
(Analogously define $\tilde T$, $d_t$ and $T$.) 
It follows that
\[
	T^n=d^n_{\tilde T^n}
	\quad\text{and}\quad
	T=d_{\tilde T}. 
\]Since $\tilde T^n$ is independent of $(d^n_t)$, 
Proposition \ref{PropositionOnWeakConvergenceOfExcursionEndpointsForGWI} 
implies the weak convergence
\[
	T^n=d^n_{\tilde T^n}\to d_{\tilde T}=T. \qedhere
\]
\end{proof}

Finally, we pass to the
\begin{proof}[Proof of Theorem \ref{TheoremOnConvergenceOfLocalTimesOfGWIProcesses}]
The strategy is the same as for the proof of Theorem \ref{randomWalkLocalTimeLimitTheorem}, 
based on inequality 
\[
	\proba{\abs{d_t-d^n_t}>\delta}
	\leq \proba{d_t-d_{\eps,t}>\delta/3}
		+\proba{\abs{d_{\eps,t}-d^n_{\eps,t}}>\delta/3}
		+\proba{d^n_t-d^n_{\eps,t}>\delta/3}
\]
(later on considered for $g^n$ and $g$.) 
For the three summands in the right-hand side, 
we prove that taking limits as $n\to\infty$ and then $\eps\to 0$ gives $0$. 
It is only in the second summand where being on the same probability space is important.  

As before, assume that the convergence of $Z^n$ to $Z$ takes place almost surely on a given probability space 
and define (besides $g_t,g^n_t,d_t$ and $d^n_t$) 
$d_{\eps,t}$ and $g_{\eps,t}$ 
(resp. $d^n_{\eps,t}$ and $g^n_{\eps,t}$) 
to be the first time after $t$ and last time before $t$ that $Z$ (resp. $Z^n$) belongs to $[0,\eps)$. 

\begin{description}
	\item[First summand] Since $d_{\eps,t}$ increases as $\eps\downarrow 0$, let $\tilde d_t$ be its limit. 
	The process $Z$ is continuous, so that $Z_{\tilde d_t}=0$ and so $\tilde d_t=d_t$. In particular, 
	\[
		\lim_{\eps\to 0}\lim_{n\to\infty}\proba{d_t-d_{\eps,t}>\delta/3}=0. 
	\]A similar argument gives us
	\[
		\lim_{\eps\to 0}\lim_{n\to\infty}\proba{g_{\eps,t}-g_{t}>\delta/3}=0. 
	\]
	\item[Second summand] We will apply Proposition \ref{SkorohodSpaceHittingTimeOfOpenSetLemma}. 
	To this end, since $Z$ is continuous, 
	it only remains to prove that the hitting time of $\set{-\eps,\eps}$ after $t$, 
	denoted $\tilde d_{t,\eps}$, actually equals $d_{t,\eps}$ when starting at $z\geq \eps$. 
	But this follows from the strong Markov property applied at time $d_{t,\eps}$, 
	and the fact that, when started at $\eps$, $Z$ immediately visits $[0,\eps)$ by regularity. 
	(This fact is contained in the proof of Theorem \ref{generalMarkovianTheorem}.) 
	We obtain, for any $\eps>0$, 
	\[
		\lim_{n\to\infty}\proba{ \abs{d_{\eps,t}-d^n_{\eps,t}} >\delta/3}=0. 
	\]For the left endpoints, we use the fact that $Z$ admits weakly continuous bridges which are invariant under time-reversal. 
	Existence and weak continuity was stated for the squared Bessel process in \cite{MR656509} or \cite[Example 2.2, p.619]{MR2789508} and holds for $Z$, while time-reversibility was analyzed in Section 5 of the first reference. 
	Then we can use the same argument as in the proof of Theorem \ref{randomWalkLocalTimeLimitTheorem}. 
	We conclude that
	\[
		\lim_{n\to\infty}\proba{ \abs{g_{\eps,t}-g^n_{\eps,t}} >\delta/3}=0. 
	\]
	\item[Third summand] We first use Lemma \ref{SkorohodSpaceHittingTimeOfOpenSetLemma} 
	to deduce that $X^n_{d^n_{\eps,t}}\to X_{d_{\eps,t}}=\eps$. 
	The hypotheses were verified in the last item. 
	Momentarily denote by $\p_x$ either the law of $X$ or $X^n$ (based on the context) 
	when started at $x$. 
	Proposition \ref{propositionOnWeakConvergenceOfGWIHittingTimesofZero} 
	then tells us that
	\[
		\imf{\p_{X^n_{d^n_{\eps,t}}}}{T^n_0>\delta/3}\to \imf{\p_{X_{d_{\eps,t}}}}{T_0>\delta/3}, 
	\]so that
	\begin{linenomath}
	\begin{align*}
		&\lim_{\eps\to 0}\lim_{n\to\infty}\proba{d^n_t-d^n_{\eps,t}>\delta/3}
		=\lim_{\eps\to 0}\lim_{n\to\infty}\esp{\imf{\p_{X^n_{d^n_{\eps,t}}}}{T_0>\delta/3}}
		\\&=\lim_{\eps\to 0}\esp{\imf{\p_{X_{d_{\eps,t}}}}{T_0>\delta/3}}
		=\lim_{\eps\to 0} \proba{d_t-d_{\eps,t}>\delta/3}\\&=0
	\end{align*}\end{linenomath}where the last equality follows from our first item. 
\end{description}
We deduce that Theorem \ref{mainLimitTheorem} is applicable and gives us the conclusions of Theorem \ref{TheoremOnConvergenceOfLocalTimesOfGWIProcesses}. 
\end{proof}

\section{Concluding remarks}
\label{SectionOnConcludingRemarks}

In this paper, we presented an invariance principle of general applicability
for the counting process of successive visits to a given state for a discrete time regenerative process. 
The limit is expressed in terms of regenerative local times 
and the hypotheses include convergence of left and right endpoints of excursions, 
which are related to first and last visits to the regenerative state. 
We showed how pathwise assumptions on the process led to simple verification of the hypotheses 
in the invariance principle. 
We also showed examples where more involved distributional methods had to be applied and provided a blueprint for how to do it. 
However, the applications were then less general that what one might suspect 
and would be more general if the following classical problems could be addressed. 

Our first problem is related to Theorem \ref{randomWalkLocalTimeLimitTheorem}. 
\begin{openproblem}
Let $X^n$ be a sequence of random walks such that $X^n_{n\cdot}/a_n$ converges weakly to a L\'evy process $X$. 
Find conditions such that the hitting times of $X^n$ started at $x$ converge to those of $X$ for any $x\in \re$. 
In our setting, where $X^1=X^n$ for all $n$, 
we solved the problem via a local limit theorem for the transition probabilites, 
which implied a local limit theorem for resolvents. 
Therefore, it is relevant to find conditions for a local limit theorem to hold. 
\end{openproblem}

Our second problem is related to Theorem \ref{TheoremOnConvergenceOfLocalTimesOfGWIProcesses}. 
\begin{openproblem}
Let $Z^n$ be a sequence of Galton-Watson processes with immigration started at $z_n$, 
where $z_n/n\to z\geq 0$ and such that $Z^n_{a_n\cdot}/n$ converges weakly to a $\cbi$ $Z$. 
Find conditions so that right and left endpoint of excursions converge. 
In our case, $Z^n=Z^1$, which had finite variance, and the scaling limit is necessarily a squared Bessel process. 
To establish convergence of endpoints of excursions, 
we passed through the convergence of hitting times of zero by Galton-Watson processes 
with the same reproduction but no immigration. 
Therefore, our first problem in this setting could be to establish 
the convergence of extinction times of sequences of GW processes. 
This can in fact be done when $Z^n=Z^1$ by means of either Kolmogorov's classical estimate 
(cf. \cite[Ch. I.10.2]{ha1}) or generalizations when $Z^1$ has power law tails using the results of  \cite{MR0228077}. 
But for general sequence, it does not seem to be immediate. 
Recall that scaling limits of GWI processes are CBI processes. 
To carry out the program of Theorem \ref{TheoremOnConvergenceOfLocalTimesOfGWIProcesses}, 
we would need the CBI process to admit a density, to build bridge laws, and to study their reversibility properties. 
Sufficient conditions for the existence of densities can be found in \cite{MR3732548}. 
We would also need local limit theorems for the approximating GWI processes. 
\end{openproblem}

Our third problem is connected to the fact that all of our examples feature one-dimensional processes. 
Indeed, the multidimensional analogues of our examples 
either have $0$ as a polar point 
or the analysis of the recurrence of zero has not been made. 
However, the class of processes introduced in \cite{MR3539299} can be recurrent in any dimension. 
These are obtained as scaling limits of inhomogeneous random walks 
with a limiting covariance structure 
and are such that their radial process is not Markovian but regenerative at zero. 
We would of course like to apply our invariance principle in this context. 

\section*{Appendix} 
The following Skorohod space lemma was used in the proof of Theorem \ref{randomWalkLocalTimeLimitTheorem}. 
We were unable to locate a reference for it. 
Let $(E,d)$ denote any metric space 
and let $D$ stand for the Skorohod space of \cadlag\ functions from $[0,\infty)$ into $E$. 
We also let $d$ stand for a metric generating the Skorohod ($J_1$) topology, 
possible confusions will be ruled out by context. 
For any $A\subset E$, define the hitting time of $A$, defined for $f\in D$, as
\[
\imf{T_A}{f}=\inf\set{t\geq 0: \imf{f}{t}\in A},
\]with the convention that the infimum of the empty set equals $\infty$. 
We will need the hitting time left-limit hitting time of $A$, by
\[
\imf{T^-_A}{f}=\inf\set{t\geq 0: \imf{f}{t-}\in A}. 
\]
\begin{lemma}
\label{SkorohodSpaceHittingTimeOfOpenSetLemma}
Let $O\subset E$ be open and let $C=\clo{O}$. 
If $f\in D$ is such that\[
\imf{T_O}{f}\leq \imf{T_C}{f}, \imf{T^-_{C}}{f}\] then 
$g\mapsto \imf{T_O}{g}$ and $g\mapsto \imf{g}{\imf{T_O}{g}}$ are continuous at $f$. 
\end{lemma}
\begin{proof}
Let $\paren{f_n}$ be a sequence of functions on Skorohod space converging to $f$ 
and let $t=\imf{T_O}{f}$ and $t_n=T_O(f_n)$. 
Then, for any $\delta>0$, there exists a continuity point  of $f$, say $t'$, belonging to $[t,t+\delta]$ 
and such that $\imf{f}{t'}\in O$. 
But then, $\imf{f_n}{t'}\to \imf{f}{t'}$, so that $\imf{f_n}{t'}\in O$ for large enough $n$. 
We deduce that $\limsup_{n} \imf{T_O}{f_n}\leq \imf{T_O}{f}$.

If $\imf{T_O}{f}=0$, then $T_O(f_n)\to T_O(f)$ and $f(T_O(f_n))\to f(T_O(f))$. 
Otherwise, note that our hypothesis says that 
for any $\delta\in (0,t)$, 
and for any $s\in [0,t-\delta]$, 
$\imf{f}{s}\not\in C$ and $\imf{f}{s-}\not\in C$. 
Hence, there exists $\eta>0$ such that $d(f(s),C), d(f(s-),C)>\eta$ for all $s\leq t-\delta$.  
(Assuming the contrary leads to $T_C(f)\leq t-\delta$ or $T^-_C(f)\leq t-\delta$. )
But then there exists $N$ such that  $\imf{d}{\imf{f_n}{s}, C}> \eta$ 
and $\imf{d}{\imf{f_n}{s-}, C}>\eta$ for any $s\in [0,t-\delta]$ and $n\geq N$. 
Indeed, this follows by letting $T>t$ 
and considering a sequence $(\lambda_n)$ of increasing homeomorphisms of $[0,T]$ into itself 
such that $f_n-f\circ\lambda_n\to 0$ uniformly on $[0,T]$. 
But then $t_n\geq t-\delta$ for $n\geq N$ and 
we deduce that $\liminf_n t_n\geq t-\delta$ for all $\delta>0$. 
By the preceding paragraph, we conclude that $t_n\to t$. 



If $f$ is continuous at $t$ then $\imf{f_n}{t_n}\to \imf{f}{t}$. 
Assume otherwise that $f$ is discontinuous at $t$. 
By hypothesis, $t\leq T^-_C(f)$, so that $\imf{f}{t-}\not \in C$. 
Then, as in our second paragraph, $f_n$ is bounded away from $C$ on $[0,\lambda_n^{-1}(t))$ 
and takes values in $O$ in any right neighborhood of $\lambda_n^{-1}(t)$ for large enough $n$. 
We then see that $t_n=\lambda_n^{-1}(t)$ for large enough $n$. 
But then
\[
f_n(T_O(f_n))=f_n(t_n)=f_n(\lambda_n^{-1}(t))\to f(t)=f(T_O(f)). \qedhere
\]
\end{proof}

\bibliography{GenBib}
\bibliographystyle{amsalpha}
\end{document}